\documentclass[10pt]{article}
\usepackage[top=3.5cm, bottom=3cm, left=3cm,right=2.5cm]{geometry} 

\usepackage{amssymb,amsmath}
\usepackage{amsthm}

\usepackage{amsfonts}
\usepackage{amscd}

\usepackage{subfigure}
\usepackage{graphicx}
\usepackage{epsfig}

\usepackage{esvect} %Vector arrows
\usepackage{enumerate}
\usepackage{enumitem} 

\usepackage{dsfont}
\usepackage{hyperref}

\newtheorem{thm}{Theorem}[section]
\newtheorem{lem}[thm]{Lemma}
\newtheorem{prop}[thm]{Proposition}
\newtheorem{cor}[thm]{Corollary}
\newtheorem{defi}[thm]{Definition}
\newtheorem{rem}[thm]{Remark}
\newtheorem{assumption}[thm]{Assumption}
\newcommand{\ga}{\alpha}
\newcommand{\gb}{\beta}

\newcommand{\eps}{\ensuremath{\varepsilon}}

\renewcommand{\gg}{\gamma}
\newcommand{\gk}{\kappa}
\newcommand{\gl}{\lambda}
\newcommand{\go}{\omega}
\newcommand{\gs}{\sigma}

\newcommand{\gD}{\Delta}
\newcommand{\gG}{\Gamma}

\newcommand{\gO}{\Omega}

\newcommand{\gT}{\Theta}

\newcommand{\cB}{\mathcal{B}}

\newcommand{\cD}{\mathcal{D}}
\newcommand{\cE}{\mathcal{E}}
\newcommand{\cF}{\mathcal{F}}

\newcommand{\cI}{\mathcal{I}}

\newcommand{\cK}{\mathcal{K}}

\newcommand{\cN}{\mathcal{N}}
\newcommand{\cO}{\mathcal{O}}
\newcommand{\cP}{\mathcal{P}}

\newcommand{\cS}{\mathcal{S}}
\newcommand{\cT}{\mathcal{T}}
\newcommand{\cU}{\mathcal{U}}
\newcommand{\cV}{\mathcal{V}}
\newcommand{\cW}{\mathcal{W}}
\newcommand{\cX}{\mathcal{X}}

\newcommand{\bE}{\mathbb{E}}

\newcommand{\bN}{\mathbb{N}}

\newcommand{\bP}{\mathbb{P}}

\newcommand{\bR}{\mathbb{R}}

\newcommand{\bT}{\mathbb{T}}

\newcommand{\bZ}{\mathbb{Z}}
\newcommand{\1}{\mathbf{1}} %Indicator function

\usepackage{color}

\newcommand*{\lrscript}[5]{{\vphantom{#1}}_{#2}^{#3}{#1}_{#4}^{#5}}
\newcommand{\dualpair}[4]{\ensuremath{\lrscript{\langle}{#1}{}{}{} #3 ,#4 \rangle_{#2}}}
\newcommand{\be}{\begin {equation}}
\newcommand{\ee}{\end  {equation}}
\newcommand{\bee}{\begin {equation*}}
\newcommand{\eee}{\end {equation*}}
\newcommand{\ol}{\overline}
\newcommand{\ul}{\underline}
\newcommand{\floor}[1]{\lfloor #1 \rfloor}

\newcommand{\indi}{\mathds{1}}
\DeclareMathOperator*{\esssup}{ess\,sup}
\DeclareMathOperator*{\essinf}{ess\,inf}

\newcommand{\KL}{{Karhunen-Lo\`{e}ve }}

\begin{document}
	
\title{Numerical analysis for time-dependent advection-diffusion problems with random discontinuous coefficients}
\author{
	Andreas Stein \footnote{University of Stuttgart, Allmandring 5b, 70569 Stuttgart, Germany} \footnote{andreas.stein@mathematik.uni-stuttgart.de}
	\and
	Andrea Barth \addtocounter{footnote}{-1}\footnotemark[\value{footnote}] 
}

\date{\today}

\maketitle

\begin{abstract}
Subsurface flows are commonly modeled by advection-diffusion equations. Insufficient measurements or uncertain material procurement may be accounted for by random coefficients. 
To represent, for example, transitions in heterogeneous media, the parameters of the equation are spatially discontinuous. Specifically, a scenario with coupled advection- and diffusion coefficients that are modeled as sums of continuous random fields and discontinuous jump components are considered. For the numerical approximation of the solution, a sample-adapted, pathwise discretization scheme based on a Finite Element approach is introduced. To stabilize the numerical approximation and accelerate convergence, the discrete space-time grid is chosen with respect to the varying discontinuities in each sample of the coefficients, leading to a stochastic formulation of the Galerkin projection and the Finite Element basis.
\end{abstract}

\section{Introduction}
In this paper we are concerned with the well-posedness of a solution to a time-dependent advection-diffusion equation with discontinuous random coefficients and its numerical discretization. The random coefficient function is modeled by a continuous part and a discontinuous part, inspired by the unique characterization of the L\'evy-Khinchine formula for L\'evy processes. We adopt this idea to spatial domains, meaning we propose jumps occurring on lower-dimensional submanifolds. The numerical discretization method has to account for these discontinuities of the coefficient functions, as otherwise (spatial) convergence rates decline. 

This work is a generalization to the elliptic setting which has drawn attention over the last decades. While many publications focus on numerical methods for continuous stochastic coefficients (see, e.g.,~\cite{ABS13, AN14, BNT07, BTZ04, BT16, BSZ11,  CGST11, CDS11, FST05, HS13, LWZ16, NT09, NTW08a, SG11, G13, GZ12}), the literature on stochastic discontinuous coefficients or stochastic interface problems is sparse (see, e.g.,~\cite{HL13, L15, Z11}). The reasons are twofold: On one hand a Gaussian random field is a well defined mathematical object and its properties are well studied, on the other hand there is no general definition and approximation method for a discontinuous (L\'evy) field. A (centered) Gaussian random field is fully characterized by its covariance operator. Discretization methods range from spectral approximations to Fourier methods (see, e.g.,~\cite{Gr18, LP11, SP07}). 
While we need an approximation for the continuous (Gaussian) part of the coefficient function, drawing samples from different jump intensity measures may also introduce a bias. Our main contribution is therefore, to provide a well-posedness result for a parabolic equation with general jump-diffusion and jump-advection coefficient and provide the analysis of a numerical approximation. Besides the approximation of the coefficient itself, we prove convergence of a pathwise sample-adapted space-time approximation. Naturally, for pathwise sample-adapted schemes, convergence rates are also random. However, in our setting an upper bound on the mean-square error can be derived but sampling has to be adopted accordingly.     

The paper is structured as follows: In Section~\ref{sec:prelim2} we state the problem and show a general existence result for pathwise solutions under mild assumptions on the data. In the following section we define the random coefficient functions and show convergence of approximations in appropriate norms. These approximations are used to develop in Section~\ref{sec:fem2} pathwise sample-adapted discretization schemes for the solution. Our main contribution is a convergence result for this approximation. We close with one- and two-dimensional numerical experiments, that confirm our theoretical findings.

%%%%%%%%%%%%%%%%%%%%%%%%%%%%%%%%%%%%%%%%%%%%%%%%%
%%%%%%%%%%%%%%%%%%%%%%%%%%%%%%%%%%%%%%%%%%%%%%%%%
\section[Initial boundary value problems]{Parabolic initial-boundary value problems and their solutions}\label{sec:prelim2}
%%%%%%%%%%%%%%%%%%%%%%%%%%%%%%%%%%%%%%%%%%%%%%%%%

Let $(\gO,\cF,\bP)$ be a complete probability space, $\bT:=[0,T]$ a time interval for some $T>0$ and $\cD\subset\bR^d,\, d\in\{1,2\}$ be a convex, polygonal domain with piecewise linear boundary.
In this paper we consider the linear, random initial-boundary value problem
\begin{equation}\label{eq:pde}
\begin{split}
\partial_t u(\go,x,t)+[Au](\go,x,t)&=f(\go,x,t)\quad\text{in $\gO\times\cD\times(0,T]$},\\
u(\go,x,0)&=u_0(\go,x)\quad\text{in $\gO\times\cD\times\{0\}$},\\
u(\go,x,t)&=0\quad\text{on $\gO\times\partial\cD\times\bT$},
\end{split}
\end{equation}
where  $f:\gO\times\cD\times\bT\to\bR$ is a random source function and $u_0:\gO\times\cD:\to\bR$ denotes the random initial condition of the partial differential equation (PDE).
Furthermore, $A$ is the second order partial differential operator given by 
\begin{equation}\label{eq:Au}
\begin{split}
[Au](\go,x,t)=-\nabla\cdot  \left(a(\go,x)\nabla u(\go,x,t)\right)+b(\go,x)\cdot\nabla u(\go,x,t)\\
\end{split}
\end{equation}
for $(\go,x,t)\in\gO\times\cD\times\bT$ with
\begin{itemize}
	\item a stochastic jump-diffusion coefficient $a:\gO\times\cD\to\bR$ and
	\item a stochastic jump-advection coefficient $b:\gO\times\cD\to\bR^d$\footnote{We could extend the above model problem by including time-dependent diffusion and/or advection 
		coefficients. If $a$ and $b$ are sufficiently smooth with respect to $t$, i.e. continuously differentiable in $\bT$, the temporal convergence rates in Subsection~\ref{subsec:temp_disc} are not affected. The focus of this article, however, is on the numerical analysis of Problem~\eqref{eq:pde} with coefficients that involve random spatial discontinuities, hence we assume for the sake of simplicity that $a$ and $b$ are time-independent.}.
\end{itemize}

We base the analysis of Problem~\eqref{eq:pde} on the standard Sobolev space $H^k(\cD)$ with the norm 
\begin{equation*}
\|v\|_{H^k(\cD)}:=\Big(\sum_{|\nu|\le k}\int_\cD |D^\nu v(x)|^2dx\Big)^{1/2}\quad\text{for $k\in\bN$},
\end{equation*}
where the $D^\nu=\partial_{x_1}^{\nu_1}\dots\partial_{x_d}^{\nu_d}$ is the mixed partial weak derivative (in space) with respect to the multi-index $\nu\in\bN_0^d$.
The seminorm corresponding to $H^k(\cD)$ is denoted by 
\begin{equation*}
|v|_{H^k(\cD)}:=\Big(\sum_{|\nu|= k}\int_\cD |D^\nu v(x)|^2dx\Big)^{1/2}.
\end{equation*}
The \textit{fractional order Sobolev spaces} $H^s(\cD)$ for $s>0$ are defined by the norm 
\begin{align*}
\|v\|_{H^s(\cD)}:=\|v\|_{H^{\floor s}(\cD)}+|v|_{H^{s-\floor s}(\cD)},\quad
|v|^2_{H^{s-\floor s}(\cD)}:=\int_\cD\int_\cD\frac{|v(x)-v(y)|^2}{|x-y|^{d+2(s-\floor s)}}dxdy,
\end{align*}
where $|\cdot|_{H^{s-\floor s}(\cD)}$ is the  the \emph{Gagliardo seminorm}, see \cite{DGV12}, and $\floor\cdot :\bR\to\bZ,\;s\mapsto\max(k\in\bZ,k\le s)$ is the \textit{floor operator}.
Further, we define $H:=L^2(\cD)$ and denote by $C$ a generic positive constant which may change from one line to another. 
Whenever necessary, the dependence of $C$ on certain parameters is made explicit. 

On the domain $\cD$, the existence of a bounded, linear operator  $\gg:H^s(\cD)\to H^{s-1/2}(\partial\cD)$
with 
\begin{equation*}
\gg:H^s(\cD)\cap C^\infty(\overline\cD)\to H^{s-1/2}(\partial\cD),\quad v\mapsto \gg v=v\vert_{\partial\cD}
\end{equation*}
and
\begin{equation}\label{eq:trace}
\|\gg v\|_{H^{s-1/2}(\partial\cD)}\le C \|v\|_{H^s(\cD)}
\end{equation}
for $s\in (1/2,3/2),\, v\in H^s(\cD)$ is ensured by the trace theorem, see for example~\cite{D96}, where $C=C(s,\cD)>0$ in Ineq.~\eqref{eq:trace} depends on the boundary of $\cD$.
Since we consider homogeneous Dirichlet boundary conditions on $\partial\cD$, we may treat $\gg$ independently of $\go$
and define the suitable solution space $V$ as 
\begin{equation*}
V:=H_0^1(\cD)=\{v\in H^1(\cD)|\;\gg v\equiv0\}, 
\end{equation*}
equipped with the $H^1(\cD)$-norm $ \|v\|_V:= \|v\|_{H^1(\cD)}$. Due to the homogeneous Dirichlet boundary conditions, 
the \emph{Poincar\'e inequality} $\|v\|_H\le C|v|_{H^1(\cD)}$ holds with $C=C(|\cD|)>0$ for all $v\in V$, where $|\cD|$ denotes the area of $\cD$.
Hence, the norms $\|\cdot\|_{H^1(\cD)}$ and $|\cdot|_{H^1(\cD)}$ are equivalent on $V$. 
Furthermore, by Jensen's inequality
\be\label{eq:nabla_norm}
\big(\sum_{i=1}^d|\partial_{x_i}v(x)|\big)^{2}\le 2^{d-1} \sum_{i=1}^d(\partial_{x_i}v(x))^2,\quad x\in\cD,
\ee
and hence $\|\sum_{i=1}^d|\partial_{x_i}v|\|^2_H\le 2^{d-1}|v|^2_{H^1(\cD)}$ for any $v\in V$.
We work on the Gelfand triplet $V\subset H\subset V'=H^{-1}(\cD)$, where $V'$ denotes the topological dual of the vector space $V$.
As the coefficients $a$ and $b$ are given by random functions, suitable solutions $u$ to Problem~\eqref{eq:pde} are in general time-dependent $V$-valued random variables. 
To investigate the integrability of $u$ with respect to $\bT$ and the underlying probability measure $\bP$ on $(\gO,\cF)$, we need to introduce the space of \textit{Bochner-integrable} functions.
\begin{defi}
	Let $(Y,\Sigma,\mu)$ be a $\sigma$-finite and complete measure space, let $(\cX,\|\cdot\|_\cX)$ be a Banach space and define the norm $\|\cdot\|_{L^p(Y;\cX)}$ for a strongly measurable $\cX$-valued function $\varphi:Y\to\cX$ by
	\begin{equation*}
	\|\varphi\|_{L^p(Y;\cX)}:=\begin{cases}
	\Big(\int_Y\|\varphi(y)\|_\cX^p\mu(dy)\Big)^{1/p}\quad\text{for $1\le p<+\infty$}\\
	\esssup\limits_{y\in Y} \|\varphi(y)\|_\cX\quad\text{for $p=+\infty$}
	\end{cases}.
	\end{equation*}
	The corresponding space of Bochner-integrable random variables is given by
	\bee
	L^p(Y;\cX):=\{\varphi:Y\to\cX\text{ is strongly measurable and }\|\varphi\|_{L^p(Y;\cX)}<+\infty\}.
	\eee
	Furthermore, the space of all continuous functions $\varphi:Y\to\cX$ is defined as
	\bee
	C(Y;\cX):=\{\varphi:Y\to\cX\text{ is continuous and }\sup_{y\in Y}\|\varphi(y)\|_\cX<+\infty\}.
	\eee
\end{defi}
We are interested in the two particular cases that 
\begin{itemize}
	\item $(Y,\Sigma,\mu)=(\bT,\cB(\bT),\mu_\bT)$, where $\cB(\bT)$ is the Borel $\gs$-algebra over $\bT$ and $\mu_\bT$ is the Lebesgue-measure on $\cB(\bT)$,
	\item $(Y,\Sigma,\mu)=(\gO,\cF,\bP)$.
\end{itemize}
The space $L^p(\gO;\cX)$ is commonly referred to as the \textit{space of Bochner-integrable random variables}.  
For any $\varphi\in L^1(\bT;\cX)$ we denote by $\partial_t\varphi\in L^1(\bT;\cX)$ the \textit{weak time derivative} of $\varphi$ if for all $\xi\in C^\infty_c(\bT;\bR)$
\bee
\int_0^T \partial_t\xi(t)\varphi(t)dt=-\int_0^T\xi(t)\partial_t\varphi(t)dt,
\eee 
where $\partial_t\xi$ is the classical (in a strong sense) time derivative of $\xi$.
The set $C^\infty_c(\bT;\bR)$ consists of all functions $\xi\in C^\infty(\bT;\bR)$ with compact support in $(0,T)$.
We record the following useful Lemma for the calculus in $L^2(\bT;H)$ (more precisely in Sec.~\ref{subsec:temp_disc}).
\begin{lem}\cite[Theorem 2, Chapter 5.9]{E10}\label{lem:a2bs_calc}
	Let $H=L^2(\cD)$ and $\varphi,\partial_t \varphi\in L^2(\bT;H)$. Then, the mapping $\varphi:\bT\to H$ is continuous, 
	\bee
	\varphi(t_2)=\varphi(t_1)+\int_{t_1}^{t_2}\partial_t \varphi(t)dt,\quad\text{for all $0\le t_1\le t_2\le T$,}
	\eee
	and it holds for $C=C(T)>0$ that 
	\bee
	\max_{t\in\bT}\|\varphi(t)\|_H^2\le C\big(\|\varphi\|^2_{L^2(\bT;H)}+\|\partial_t \varphi\|^2_{L^2(\bT;H)}\big).
	\eee
\end{lem}

\begin{rem} \label{rem:para}
	We may as well consider non-homogeneous boundary conditions, that is $u(\go,x,t)=g_1(\go,x,t)$ for $g_1:\gO\times\partial\cD\times\bT\to\bR$.
	The corresponding trace operator $\gg$ is still well defined provided that $g_1(\go,\cdot,\cdot)$ can be extended almost surely to a function $\widetilde g_1(\go,\cdot,\cdot)\in L^1(\bT;H^1(\cD))$ with $\partial_t \widetilde g_1(\go,\cdot,\cdot)\in L^1(\bT;H^{-1}(\cD))$.
	Then, $u-\widetilde g_1\in L^1(\bT;V)$ may be regarded as a solution to the modified problem
	\begin{align*}
	\partial_t (u-\widetilde g_1)(\go,x,t)+[A(u-\widetilde g_1)](\go,x,t)&=f(\go,x,t)-[A\widetilde g_1](\go,x,t)
	-\partial_t\widetilde g_1(\go,x,t)\quad\text{on $\gO\times\cD\times\bT$},\\
	(u-\widetilde g_1)(\go,x,0)&=u_0(\go,x)
	-\widetilde g_1(\go,x,0)\quad\text{on $\gO\times\cD\times\{0\}$,\quad and}\\
	(u-\widetilde g_1)(\go,x,t)&=0 \quad\text{on $\gO\times\partial\cD\times\bT$.}
	\end{align*}
	But this is in fact a version of Problem~\eqref{eq:pde} with modified source term and initial value (see also~\cite[Chapter 6.1]{E10}).
\end{rem}

We introduce the bilinear form associated to $A$ in order to derive a weak formulation of the initial-boundary value Problem~\eqref{eq:pde}.
For fixed $\go\in\gO$ and $t\in\bT$, multiplying Eq.~\eqref{eq:pde} with a test function $v\in V$ and integrating by parts yields the variational equation
\be\label{eq:var2}
\int_{\cD}\partial_tu(\go,x,t)v(x)dx +B_\go(u(\go,\cdot,t),v)=F_{\go,t}(v).
\ee
The bilinear form $B_{\go}:V\times V\to\bR$ is given by  
\begin{align*}
B_\go(u,v)&=\int_\cD a(\go,x)\nabla u(x)\cdot\nabla v(x)+b(\go,x)\cdot\nabla u(x)v(x)dx
=(a(\go,\cdot),\sum_{i=1}^d \partial_{x_i}u\,\partial_{x_i}v)+(b(\go,\cdot)\cdot\nabla u,v),
\end{align*}
where $(\cdot,\cdot)$ denotes the $L^2(\cD)$-scalar product.
The source term is transformed into the right hand side functional
\begin{equation*}
F_{\go,t}:V\to\bR,\quad v\mapsto\int_\cD f(\go,x,t)v(x)dx,
\end{equation*}
and the integrals with respect to $\partial_tu$ and $f$ are understood as the duality pairings 
\begin{align*}
\int_\cD \partial_t u(\go,x,t)v(x)dx&=\dualpair{V'}{V}{\partial_t u(\go,\cdot,t)}{v},\\
\int_\cD f(\go,x,t)v(x)dx&=\dualpair{V'}{V}{f(\go,\cdot,t)}{v}.
\end{align*}

\begin{defi}
	For fixed $\go\in\gO$, the \textit{pathwise weak solution} to Problem~\eqref{eq:pde} is a function $u(\go,\cdot,\cdot)\in L^2(\bT;V)$ with $\partial_t u(\go,\cdot,\cdot)\in L^2(\bT;V')$ such that for $t\in\bT$ and all $v\in V$,
	\bee
	\dualpair{V'}{V}{\partial_t u(\go,\cdot,t)}{v}+B_\go(u(\go,\cdot,t),v)=F_{\go,t}(v),\quad u(\go,\cdot,0)=u_0(\go,\cdot).
	\eee
\end{defi}

The following assumptions allow us to show existence and uniqueness of a pathwise weak solution to Eq.~\eqref{eq:pde} and guarantee measurability of the solution map $u:\gO\to L^2(\bT;V)$.\pagebreak
\begin{assumption}\label{ass:para}
	~
	\begin{enumerate}[label=(\roman*)]
		\item For each $x\in\cD$, the mappings $\go\mapsto a(\go,x)$ and $\go\mapsto b(\go,x)$ are measurable.
		\item For all $\go\in\gO$ it holds that 
		\bee
		a_-(\go):=\essinf\limits_{x\in\cD}\, a(\go,x)>0,\quad a_+(\go):=\esssup\limits_{x\in\cD}\, a(\go,x)<+\infty.
		\eee
		\item  $f\in L^p(\gO;L^2(\bT;V')), u_0\in L^p(\gO;H)$ and $1/a_-\in L^q(\gO;\bR)$, for some $p,q\in[1,\infty]$ such that $1/p+1/q\le 1$.
		\item There are constants $\ol b_1,\ol b_2\ge0$ such that $\|b(\go,x)\|_\infty\le\min(\ol b_1a(\go,x),\ol b_2)$ holds for almost all $\go\in\gO$ and almost all $x\in\cD$. Here $\|\cdot\|_\infty$ denotes the supremum norm in $\bR^d$.
	\end{enumerate}
\end{assumption}

\begin{rem}\label{rem:measurability}
	Assumption~\ref{ass:para}(i,ii) implies measurability of the random variables $a_-,a_+:\gO\to\bR$. For instance, $a_+(\go)=\|a(\go,\cdot)\|_{L^\infty(\cD)}$ may be written as the point-wise limit of the measurable functions $\|a(\go,\cdot)\|_{L^n(\cD)}$ for $n\to\infty$, see e.g. \cite[Lemma 13.1]{AB06}.
\end{rem}

\begin{thm}\label{thm:exis}
	For any $w\in L^2(\bT;V)$ define the (pathwise) parabolic norm 
	\bee
	\|w\|_{*,t}:=\Big(\|w(\cdot,t)\|_H^2+\int_0^t |w(\cdot,z)|_{H^1(\cD)}^2dz\Big)^{1/2},\quad t\in\bT.
	\eee
	Under Assumption~\ref{ass:para}, for any $\go\in\gO$, there exists a unique pathwise weak solution $u(\go,\cdot,\cdot)\in L^2(\bT;V)\cap C(\bT;H)$ to Problem~\eqref{eq:pde} and $u:\gO\to L^2(\bT;V),\, \go\mapsto u(\go,\cdot,\cdot)$ is strongly measurable.  
	Further, for any $r\in[1,(1/p+1/q)^{-1}]$
	\be \label{eq:e_est} 
	\begin{split}
		\bE\Big(\sup_{t\in\bT} \|u\|^r_{*,t}\Big)^{1/r}
		&\le C (1+\|1/a_-\|_{L^{q}(\gO;\bR)}) \Big(\|u_0\|_{L^p(\gO;H)}
		+\|f\|_{L^p(\gO;L^2(\bT;V'))}\Big)<+\infty,
	\end{split}
	\ee
	with $C=C(\ol b,T,q)>0$.
	Moreover, if  $f\in L^p(\gO;L^2(\bT;H))$, then for any $r\in[1,(1/p+(1/(2q))^{-1}]$
	\bee
	\begin{split}
		\bE\Big(\sup_{t\in\bT} \|u\|^r_{*,t}\Big)^{1/r}&\le C (1+\|1/a_-\|_{L^{q}(\gO;\bR)}^{1/2}) \Big(\|u_0\|_{L^p(\gO;H)}+\|f\|_{L^p(\gO;L^2(\bT;H))}\Big)<+\infty.
	\end{split}
	\eee
\end{thm}

\begin{proof}
	For fixed $\go\in\gO$,  the bilinear form $B_\go: V\times V\to\bR$ in Eq.~\eqref{eq:var2} is continuous and coercive by Assumption~\ref{ass:para}. 
	Hence, existence and uniqueness of a pathwise weak solution $u(\go,\cdot,\cdot)\in L^2(\bT;V)\cap C(\bT;H)$ to Problem~\eqref{eq:pde} follows as for deterministic parabolic problems, see for instance \cite[Chapter 7.1]{E10} or \cite[Chapter 11]{QV97}. 
	
	Now define the space $\cX:=L^2(\bT;V)\times L^2(\bT;V')$ with norm $\|(y_1,y_2)\|_\cX:=\|y_1\|_{L^2(\bT;V)}+\|y_2\|_{L^2(\bT;V')}$ and note that the mapping $\gO\to \cX,\,\go\mapsto (u(\go,\cdot,\cdot),\partial_tu(\go,\cdot,\cdot))$ is well-defined. 
	Let $(v_i,i\in\bN)\subset V$ be a basis of $V$ and for fixed $t\in\bT$ and $i\in\bN$ define the functional 
	\bee
	J_i:\gO\times \cX,\; (\go,w)\mapsto \int_0^\bT B_\go(w(\cdot,t) ,v_i)-F_{\go,t}(v_i)+\dualpair{V'}{V}{\partial_t w(\cdot,t)}{v_i}dt.
	\eee
	By Assumption~\ref{ass:para}, it follows that $J_i$ is a \textit{Carath\'eodory map}, i.e. measurable in $\gO$ and continuous in $\cX$, and thus $\cF\otimes\cB(\cX)-\cB(\bR)$-measurable.
	The separability of $L^2(\bT;V)$ and $L^2(\bT;V')$ entails separability of $\cX$ and, furthermore, $\cB(\cX)=\cB(L^2(\bT;V))\otimes \cB(L^2(\bT;V'))$.	
	To show the measurability of $u$, we define the correspondence 
	\bee
	\varphi_i(\go):=\{w\in \cX|\, J_i(\go,w)=0\}.
	\eee
	By	\cite[Corollary 18.8]{AB06} the graph $\text{Gr}(\varphi_i):=\{(\go,w)\in\gO\times \cX|\, w\in\varphi_i(\go)\}$ is measurable, i.e. $\text{Gr}(\varphi_i)\in\cF\otimes\cB(\cX)$.
	Since this yields 
	\bee
	\{(\go,u(\go,\cdot,\cdot),\partial_tu(\go,\cdot,\cdot))|\,\go\in\gO\}=\bigcap_{i\in\bN}\text{Gr}(\varphi_i)\in\cF\otimes\cB(\cX),
	\eee 
	the mapping $\go\to (u(\go,\cdot,\cdot),\partial_tu(\go,\cdot,\cdot))$ is $\cF-\cB(\cX)$-measurable (see e.g. \cite[Theorem 18.25]{AB06}).
	As $\cB(\cX)=\cB(L^2(\bT;V))\otimes \cB(L^2(\bT;V'))$, the marginal mappings $u:\gO\to L^2(\bT;V)$ and $\partial_tu:\gO\to L^2(\bT;V')$ are strongly $\cF-\cB(L^2(\bT;V))$-measurable and $\cF-\cB(L^2(\bT;V'))$-measurable, respectively.
	We note that it is sufficient to test against a basis of $V$ in order to obtain the measurability of the $L^2(\bT;V')$-valued map $\partial_t u$, since the embeddings $V\subset H\subset V'$ are dense.
	
	To show the estimate~\eqref{eq:e_est}, we fix $\go\in\gO,\, t\in\bT$, test against $v=u(\go,\cdot,t)\in V$ in Eq.~\eqref{eq:var2} and obtain
	\bee
	\dualpair{V'}{V}{\partial_t u(\go,\cdot,t)}{u(\go,\cdot,t)}+B_\go(u(\go,\cdot,t),u(\go,\cdot,t))=F_{\go,t}(u(\go,\cdot,t)).
	\eee
	As $u(\go,\cdot,\cdot)\in L^2(\bT;V)$ it holds that 
	\bee
	\dualpair{V'}{V}{\partial_t u(\go,\cdot,t)}{u(\go,\cdot,t)}=
	\frac{1}{2}\frac{d}{dt} \|u(\go,\cdot,t)\|_H^2,
	\eee
	see i.e. \cite[Chapter 5.9]{E10}. 
	Rearranging the terms yields
	\be
	\begin{split} \label{eq:energy1}
		\frac{1}{2}\frac{d}{dt}\|u(\go,\cdot,t)\|_H^2+(a(\go,\cdot),\sum_{i=1}^d(\partial_{x_i}u(\go,\cdot,t))^2)
		&=-(b(\go,\cdot) \cdot\nabla  u(\go,\cdot,t),u(\go,\cdot,t))+F_{\go,t}(u(\go,\cdot,t))\\
		&=:I+II.
	\end{split}
	\ee
	The first term is bounded with Young's inequality, Assumption~\ref{ass:para} and Ineq.~\eqref{eq:nabla_norm} via
	\begin{align*}
	I&\le \frac{2^{1-d}}{4\ol b_1}\|\|b(\go,\cdot)\|^{1/2}_\infty\sum_{i=1}^d|\partial_{x_i}u(\go,\cdot,t)|\|_H^2
	+2^{d-1}\ol b_1\|\|b(\go,\cdot)\|^{1/2}_\infty u(\go,\cdot,t)\|_H^2\\
	&\le \frac{1}{4}(a(\go,\cdot),\sum_{i=1}^d(\partial_{x_i}u(\go,\cdot,t))^2)+2^{d-1}\ol b_1\ol b_2\|u(\go,\cdot,t)\|_H^2.
	\end{align*}
	By the Poincar\'e inequality it holds that $\|u\|^2_V=|u|^2_{H^1(\cD)}+\|u\|^2_H\le (1+C^2)|u|^2_{H^1(\cD)}$ and we estimate $II$ by
	\begin{align*}
	II&\le (1+C^2)\frac{\|f(\go,\cdot,t)\|_{V'}^2}{a_-(\go)}+ \frac{a_-(\go)}{4(1+C^2)}\|u(\go,\cdot,t)\|_V^2\\
	&\le (1+C^2)\frac{\|f(\go,\cdot,t)\|_{V'}^2}{a_-(\go)}+ \frac{a_-(\go)}{4}|u(\go,\cdot,t)|_{H^1(\cD)}^2\\
	&\le (1+C^2)\frac{\|f(\go,\cdot,t)\|_{V'}^2}{a_-(\go)}+ \frac{1}{4}(a(\go,\cdot),\sum_{i=1}^d(\partial_{x_i}u(\go,\cdot,t))^2).
	\end{align*}
	Hence, Eq.~\eqref{eq:energy1} implies
	\begin{align*}
	\frac{d}{dt}\|u(\go,\cdot,t)\|_H^2+(a(\go,\cdot),\sum_{i=1}^d(\partial_{x_i}u(\go,\cdot,t))^2)
	\le C\Big(\frac{\|f(\go,\cdot,t)\|_{V'}^2}{a_-(\go)}+\|u(\go,\cdot,t)\|_H^2\Big).
	\end{align*}
	We integrate over $\bT$ and use Gr\"onwall's inequality to obtain
	\begin{align*}
	\|u(\go,\cdot,t)\|_H^2+a_-(\go)\int_0^t|u(\go,\cdot,z)|_{H^1(\cD)}^2dz
	&\le \|u(\go,\cdot,t)\|_H^2+\int_0^t(a(\go,\cdot),\sum_{i=1}^d(\partial_{x_i}u(\go,\cdot,z))^2)dz\\
	&\le \exp(CT)\Big(\|u_0(\go,\cdot)\|_H^2+\frac{\|f(\go,\cdot,\cdot)\|_{L^2(\bT;V')}^2}{a_-(\go)}\Big),
	\end{align*}
	where we emphasize that the last estimate is independent of $t$.
	If $a_-(\go)\le1$ holds for fixed $\go$,
	\begin{align*}
	\sup_{t\in\bT}\|u(\go,\cdot,\cdot)\|^2_{*,t}
	&=\sup_{t\in\bT}\Big(\|u(\go,\cdot,t)\|_H^2+\int_0^t|u(\go,\cdot,z)|_{H^1(\cD)}^2dz\Big)\\
	&\le \exp(CT)\left(\frac{\|u_0(\go,\cdot)\|_H^2+\|f(\go,\cdot,\cdot)\|_{L^2(\bT;V')}^2}{a^2_-(\go)}\right).
	\end{align*}
	On the other hand, if $a_-(\go)>1$, it follows that
	\begin{align*}
	\sup_{t\in\bT}\|u(\go,\cdot,\cdot)\|^2_{*,t}\le 
	\exp(CT)\big(\|u_0(\go,\cdot)\|_H^2+\|f(\go,\cdot,\cdot)\|_{L^2(\bT;V')}^2\big).
	\end{align*}
	With the inequalities $\sqrt{c_1+c_2}\le \sqrt{c_1}+\sqrt{c_2}$ and $(c_1+c_2)^r\le 2^{r-1}(c_1^r+c_2^r)$ for $c_1,c_2\ge0,r\ge1$, and by taking expectations this yields for any $r\in[1,(1/p+1/q)^{-1}]$
	\begin{align*}
	\bE\Big(\sup_{t\in\bT} \|u\|^r_{*,t}\Big)^{1/r}
	&\le C\bE\Big(\frac{\|u_0\|_H^r+\|f\|^r_{L^2(\bT;H)}}{a^r_-}\indi_{\{a_-\le 1\}}+(\|u_0\|_H^r+\|f\|^r_{L^2(\bT;V')})\indi_{\{a_->1\}}\Big)^{1/r}\\
	&\le C(1+\|1/a_-\|_{L^q(\gO;\bR)})\Big(\|u_0\|_{L^p(\gO;H)}+\|f\|_{L^p(\gO;L^2(\bT;V'))}\Big),
	\end{align*}
	where we have used Assumption~\ref{ass:para} and H\"older's inequality for the last estimate.
	
	For the second part of the claim, given that $f\in L^p(\gO;L^2(\bT;H))$, we may bound $II$ via
	\bee
	II\le \frac{1}{2}\|f(\go,\cdot,t)\|_{H}^2+ \frac{1}{2}\|u(\go,\cdot,t)\|_H^2
	\eee
	and proceed as for the first term, using Gr\"onwall's inequality, to obtain
	\begin{align*}
	\|u(\go,\cdot,t)\|_H^2+a_-(\go)\int_0^t|u(\go,\cdot,z)|_{H^1(\cD)}^2dz
	\le C\Big(\|u_0(\go,\cdot)\|_H^2+\|f(\go,\cdot,\cdot)\|_{L^2(\bT,H)}^2\Big).
	\end{align*}
	Finally, with H\"older's inequality it follows for any $r\in[1,(1/p+1/(2q))^{-1}]$ that
	\begin{align*}
	\bE\Big(\sup_{t\in\bT} \|u\|^r_{*,t}\Big)^{1/r}\le 
	C (1+\|1/a_-\|_{L^q(\gO;\bR)}^{1/2})\Big(\|u_0\|_{L^p(\gO;H)}+\|f\|_{L^p(\gO;L^2(\bT;H))}\Big).
	\end{align*}
	
\end{proof}

To incorporate discontinuities at random submanifolds of $\cD$, we introduce the jump-diffusion coefficient $a$ and jump-advection coefficient $b$ in the subsequent section.
The introduced coefficients allow us to derive well-posedness- and regularity results based on Theorem~\ref{thm:exis} for the solution to the parabolic problem with discontinuous coefficients.

\section[Random discontinuous problems]{Random parabolic problems with discontinuous coefficients}\label{sec:jump_diffusion}
To obtain a stochastic jump-diffusion coefficient representing the permeability in a subsurface flow model, we use the random coefficient $a$ from the elliptic diffusion problem in~\cite{BS18b} consisting of a (spatial) Gaussian random field with additive discontinuities on random submanifolds of $\cD$. The specific structure of $a$ may be utilized to model the hydraulic conductivity within heterogeneous and/or fractured media and is thus considered time-independent (see also Remark~\ref{rem:para}).
The advection term in this model should then be driven by the same random field and inherit the same discontinuous structure as the diffusion term. Thus, we consider the coefficient $b$ as an essentially linear mapping of $a$.
Since the coefficients usually involve infinite series expansions in the Gaussian field and/or sampling errors in the jump measure, 
we further describe how to obtain tractable approximations of $a$ and $b$.
Subsequently, existence and stability results for weak solutions of the unapproximated resp. approximated parabolic problems based on Theorem~\ref{thm:exis} are proved.
We conclude this section by showing that the approximated solution converges to the solution $u$ of the (unapproximated) advection-diffusion problem in a suitable norm.

\subsection{Jump-diffusion coefficients and their approximations}
\begin{defi}\label{def:a2}
	The \textit{jump-diffusion coefficient} $a$ is defined as
	\begin{equation*}
	a:\gO\times\cD\to\bR_{>0},\quad (\go,x)\mapsto\ol a(x)+\Phi(W(\go,x))+P(\go,x),
	\end{equation*}
	where
	\begin{itemize}
		\item $\ol a\in C^1(\ol\cD;\bR_{\ge0})$ is non-negative, continuous and bounded.
		\item $\Phi\in C^1(\bR;\bR_{>0})$ is a continuously differentiable, positive mapping.
		\item $W\in L^2(\gO;H)$ is a zero-mean Gaussian random field. Associated to $W$ is a non-negative, symmetric trace class operator $Q:H\to H$.
		\item $\cT:\gO\to\cB(\cD),\;\go\mapsto\{\cT_1,\dots,\cT_{\tau}\}$ is a random partition of $\cD$, i.e. the $\cT_i$ are disjoint open subsets of $\cD$ such that $|\cT_i|>0$ for $i=1,\dots,\tau(\go)$ and $\overline\cD=\bigcup_{i=1}^\tau\ol\cT_i$.
		The number $\tau$ of elements in $\cT$ is a random variable $\tau:\gO\to\bN$ on $(\gO,\cF,\bP)$. 
		Associated to $\cT$ is a measure $\gl$ on $(\cD,\cB(\cD))$ that controls the position of the random elements $\cT_i$. 
		\item $(P_i, i\in\bN)$ is a sequence of non-negative random variables on $(\gO,\cF,\bP)$ and
		$$P:\gO\times\cD\to\bR_{\ge0},\quad(\go,x)\mapsto\sum_{i=1}^{\tau(\go)}\indi_{\{\cT_i\}}(x)P_i(\go).$$
		The sequence $(P_i, i\in\bN)$ is independent of $\tau$ (but not necessarily i.i.d.).
		
	\end{itemize}
	Based on $a$, the \textit{jump-advection coefficient} $b$ is given for vector fields $\widetilde b_1, \widetilde b_2\in L^\infty(\cD)^d$ by 
	\begin{equation*}
	b:\gO\times\cD\to\bR^d,\quad (\go,x)\mapsto \min (a(\go,x)\widetilde b_1(x), \widetilde b_2(x)). 
	\end{equation*}
	
\end{defi}

\begin{rem}\label{rem:advection}
	The definition of the jump-advection coefficient immediately implies Assumption~\ref{ass:para}(iv) since
	$$\|b(\go,x)\|_\infty\le\min(\ol b_1a(\go,x),\ol b_2)$$
	holds with suitable constants $\ol b_1,\ol b_2>0$ for almost all $\go\in\gO$  and almost all $x\in\cD$ .
	The upper bound with respect to $\ol b_2$ is due to technical reasons and not restrictive in practical applications, as $\ol b_2$ may be arbitrary large.
\end{rem}

In general, the structure of $a$ as in Def.~\ref{def:a2} does not allow us to draw samples from the exact distribution of this random function.
We remark that $\gl$ may be used to concentrate the submanifolds that generate $\cT$ on certain areas in $\cD$, see Section~\ref{sec:num2} for examples.
The Gaussian random field may be approximated by truncated \KL expansions:
Let $((\eta_i,e_i), i\in\bN)$ denote the sequence of eigenpairs of $Q$, where $Q:H\to H$ is the covariance operator of the Gaussian field $W$ and the eigenvalues are given in decaying order $\eta_1\ge\eta_2\ge\dots\ge0$.
Since $Q$ is trace class, the Gaussian random field $W$ admits the representation 
\be\label{eq:KL_expansion}
W=\sum_{i\in\bN}\sqrt{\eta_i}e_iZ_i,
\ee
where $(Z_i,i\in\bN)$ are independent standard normally distributed random variables.
The series above converges in $L^2(\gO;H)$ and almost surely (see e.g.~\cite{BL12}).
The truncated \KL expansion $W_N$ of $W$ is then given by 
\be\label{eq:KL_expansion_trunc}
W_N:=\sum_{i=1}^N\sqrt{\eta_i}e_iZ_i,
\ee
where we call $N\in\bN$ the \textit{cut-off index} of $W_N$.
In addition, it may be possible that the sequence of jumps $(P_i, i\in\bN)$ cannot be sampled exactly but only with an intrinsic bias (see \cite[Remark 3.4]{BS18b}).
The biased samples are denoted by $(\widetilde P_i, i\in\bN)$ and the error which is induced by this approximation is represented by the parameter $\eps>0$ (see Assumption~\ref{ass:EV2}).
To approximate $P$ using the biased sequence $(\widetilde P_i, i\in\bN)$ instead of $(P_i, i\in\bN)$ we define 
\bee
P_\eps:\gO\times\cD\to\bR,\quad(\go,x)\mapsto\sum_{i=1}^{\tau(\go)}\indi_{\{\cT_i\}}(x)\widetilde P_i(\go).
\eee
The \textit{approximated jump-diffusion coefficient} $a_{N,\eps}$ is then given by
\be \label{eq:a_approx2}
a_{N,\eps}(\go,x):=\ol a(x)+\Phi(W_N(\go,x))+P_\eps(\go,x),
\ee
and the \textit{approximated jump-advection coefficient} $b_{N,\eps}$ via
\bee
b_{N,\eps}(\go,x):=\min (a_{N,\eps}(\go,x)\widetilde b_1(x), \widetilde b_2(x)).
\eee
Substituting the approximated jump coefficients into the parabolic model Problem~\eqref{eq:pde} yields  
\begin{align}\label{eq:pde_approx}
\begin{split}
\partial_t u_{N,\eps}(\go,x,t)+[A_{N,\eps}u_{N,\eps}](\go,x,t)&=f(\go,x,t)\quad\text{in $\gO\times\cD\times (0,T]$},\\
u_{N,\eps}(\go,x,0)&=u_0(\go,x)\quad\text{in $\gO\times\cD\times\{0\}$},\\
u_{N,\eps}(\go,x)&=0\quad\text{on $\gO\times\partial\cD$},
\end{split}
\end{align}
where the approximated second order differential operator $A_{N,\eps}$ is given by 
\bee
[A_{N,\eps}u](\go,x,t)=-\nabla\cdot \left(a_{N,\eps}(\go,x)\nabla u(\go,x,t)\right)+b_{N,\eps}(\go,x)\cdot\nabla  u(\go,x,t).
\eee
The pathwise variational formulation of Eq.~\eqref{eq:pde_approx} is then analogous to Eq.~\eqref{eq:var2} given by: 
For fixed $\go\in\gO$ with given $f(\go,\cdot)$, find $u_{N,\eps}(\go,\cdot,\cdot)\in L^2(\bT;V)$ with $\partial_t u_{N,\eps}(\go,\cdot,\cdot)\in L^2(\bT;V')$ such that it holds, for $t\in\bT$ and for all $v\in V$
\be\label{eq:var2_approx}
\dualpair{V'}{V}{\partial_t u_{N,\eps}(\go,\cdot,t)}{v}+B_\go^{N,\eps}(u_{N,\eps}(\go,\cdot,t),v)=F_{\go,t}(v).
\ee
The approximated bilinear form is given for $v,w\in V$ by
\begin{align*}
B^{N,\eps}_\go(v,w)=\int_\cD a_{N,\eps}(\go,x)\nabla v(x)\cdot\nabla w(x)+b_{N,\eps}(\go,x)\cdot\nabla   v(x)w(x)dx.
\end{align*}

The following assumptions guarantee that we can apply Theorem~\ref{thm:exis} also in the jump-diffusion setting and that therefore pathwise solutions $u$ and $u_{N,\eps}$ exist.

\begin{assumption}\label{ass:EV2}
	~
	\begin{enumerate}[label=(\roman*)]
		\item The eigenfunctions $e_i$ of $Q$ are continuously differentiable on $\cD$ and there exist constants $\ga,\gb,C_e,C_\eta>0$ such that for any $i\in\bN$
		\bee
		\|e_i\|_{L^\infty(\cD)}\le C_e,\quad\max_{j=1,\dots,d}\|\partial_{x_j} e_i\|_{L^\infty(\cD)}\le C_ei^\alpha\quad\text{and}\quad\sum_{i=1}^\infty\eta_ii^\gb\leq C_\eta<+\infty.
		\eee 
		
		\item Furthermore, the mapping $\Phi$ as in Definition~\ref{def:a2} and its derivative are bounded for $w\in\bR$ by
		\bee
		\phi_1\exp(\phi_2 w)\ge \Phi(w)\ge\phi_{1}\exp(-\phi_2 w),\quad|\frac{d}{dx}\Phi(w)|\le\phi_3\exp(\phi_4|w|),
		\eee
		where $\phi_1,\dots,\phi_4>0$ are arbitrary constants.
		\item There exists $p>1$ such that $f\in L^p(\gO;L^2(\bT;V'))$ and $u_0\in L^p(\gO;H)$. 
		\item The sequence $(P_i, i\in\bN)$ consists of nonnegative and bounded random variables $P_i\in [0,\ol P]$ for some $\ol P>0$. In addition, for $s>1$ such that $1/p+1/s<1$ there exists a sequence of approximations $(\widetilde P_i, i\in\bN)\subset[0,\ol P]^\bN$ so that the sampling error is bounded, for some $\eps>0$, by
		\begin{equation*}
		\bE(|\widetilde P_i-P_i|^s)\le\eps,\quad i\in\bN.
		\end{equation*}	
	\end{enumerate}
\end{assumption}

\begin{rem}
	The exponential bounds on $\Phi$ and its derivative imply that $u\in L^r(\gO;L^2(\bT;V))$ for any $r\in[1,p)$. That is, the integrability of $u$ with respect to $\gO$ only depends on the stochastic regularity of $f$ and $u_0$. In fact, Theorem~\ref{thm:exis} shows that far weaker assumptions on $a$ (resp. $\Phi$) are possible to achieve $u\in L^r(\gO;L^2(\bT;V))$, at the cost that $r$ then also depends on the integrability of $a_-$. At this point we refer to~\cite{BS18b}, where the regularity of an elliptic diffusion problem with $a$ as in Definition~\ref{def:a2}, but less restricted functions $\Phi$ and $P$ is investigated.
	However, Assumption~\ref{ass:EV2} includes the important case that $\Phi(W)$ is a log-Gaussian random field and the bounds on $\Phi$ are merely imposed for a clear and simplified presentation of the results. 
	On a further note, the assumptions on the eigenpairs $((\eta_i,e_i),i\in\bN)$ are natural and include the case that $Q$ is a Mat\'ern-type or Brownian-motion-type covariance function.
\end{rem}

\begin{lem}\label{lem:a2}
	Let $a$ and $b$ be as in Definition~\ref{def:a2}, let $a_{N,\eps}$ and $b_{N,\eps}$ given by Eq.~\eqref{eq:a_approx2}, and let Assumption~\ref{ass:EV2} hold. 
	Then, each pair $(a,b)$ and $(a_{N,\eps},b_{N,\eps})$ satisfies Assumption~\ref{ass:para}(i) and~(ii).
	
	Moreover, define the real-valued random variables
	\begin{align*}
	&a_-:=\essinf\limits_{x\in\cD} a(\go,x),\quad a_{N,\eps,-}:=\essinf\limits_{x\in\cD} a_{N,\eps}(\go,x),\\
	&a_+:=\esssup\limits_{x\in\cD} a(\go,x),\quad a_{N,\eps,+}:=\esssup\limits_{x\in\cD} a_{N,\eps}(\go,x).
	\end{align*}
	Then, $1/a_-,1/a_{N,\eps,-},a_+,a_{N,\eps,+}\in L^q(\gO;\bR)$ for any $q\in [1,\infty)$ and there exists $C=C(q,\phi_1,\phi_2)>0$, independent of $N$ and $\eps$, such that 
	\begin{equation*}
	\|1/a_-\|_{L^q(\gO;\bR)},\;\|1/a_{N,\eps,-}\|_{L^q(\gO;\bR)},\;\|a_{+}\|_{L^q(\gO;\bR)},\;\|a_{N,\eps,+}\|_{L^q(\gO;\bR)}\le C<+\infty.
	\end{equation*}
\end{lem}
\begin{proof}
	By Definition~\ref{def:a2}	
	\begin{equation*}
	a(\go,x)=\ol a(x)+\Phi(W(\go,x))+P(\go,x),
	\end{equation*}
	for random fields $W$ and $P$, hence the mapping $\go\mapsto a(\go,x)$ is $\cF-\cB(\bR)$-measurable for any $x\in\cD$.
	With this, the measurability of $b$ follows immediately.
	Since $\ol a$ and $P$ are nonnegative, and $\Phi\circ W(\cdot,x):\gO\to (0,+\infty)$ for all $x\in\cD$, we have that $a_-:\gO\to(0,+\infty)$.
	On the other hand, $\ol a$ and $P$ are bounded mappings by Definition~\ref{def:a2} and therefore $a_+(\go)<+\infty$ for all $\go\in\gO$.  
	For fixed parameters $N\in\bN$ and $\eps>0$, the assertion for $(a_{N,\eps},b_{N,\eps})$ follows analogously.
	By Remark~\ref{rem:measurability}, we also observe that $a_{N,\eps,-},a_{N,\eps,+}:\gO\to\bR$ are measurable mappings.
	To bound $\|1/a_-\|_{L^q(\gO;\bR)}$, we use that $W$ and $W_N$ are centered, almost surely bounded Gaussian random fields (\cite[Lemma 3.5]{BS18b}) on $\cD$ which implies 
	$E:=\bE(\sup_{x\in\cD} W(x))<+\infty$
	as well as
	\be\label{eq:tail2}
	\bP(\sup_{x\in\cD} W(\cdot,x)-E\ge c)\le \exp(-\frac{c^2}{2\ol\sigma^2})
	\ee
	for all $c>0$ and $\overline\sigma^2:=\sup_{x\in\cD}\bE(W(\cdot,x)^2)\le tr(Q)$.
	Furthermore, by the symmetry of $W$,
	\be\label{eq:W_symm}
	\bP(||W(x)||_{L^\infty(\cD)}>c)\le 2\bP(\sup_{x\in\cD} W(\cdot,x)>c).
	\ee
	With Assumption~\ref{ass:EV2}~(ii), and since $$||\exp(|W|)||_{L^\infty(\cD)}\le \exp(||W||_{L^\infty(\cD)}),$$ we then obtain for arbitrary $q\in[1,\infty)$
	\begin{align*}
	\bE(1/a_-^q)&\le\bE\big(\big(\inf_{x\in\cD}\Phi(W(\cdot,x)\big)^{-q}\big)\\
	&=\bE\big(\sup_{x\in\cD}\Phi(W(\cdot,x))^{-q}\big)\\
	&\le \frac{1}{\phi_1^q}\bE(\sup_{x\in\cD}\exp(q\phi_2|W(\cdot,x)|))\\
	&\le \frac{1}{\phi_1^q}\bE(\exp(q\phi_2||W||_{L^\infty(\cD)})).
	\end{align*}
	By Fubini's Theorem, integration by parts and Ineqs.~\eqref{eq:W_symm},~\eqref{eq:tail2} this yields
	\begin{align*}
	\bE(\exp(q\phi_2||W||_{L^\infty(\cD)}))&=\int_0^\infty q\phi_2 \exp(q\phi_2c)\bP(||W||_{L^\infty(\cD)}>c)dc\\
	&\le q\phi_2 \exp(q\phi_2E) + 2\int_E^\infty q\phi_2 \exp(q\phi_2c))\bP(\sup_{x\in\cD} W(\cdot,x)>c)dc\\
	&\le q\phi_2 \exp(q\phi_2E) + 2\int_E^\infty q\phi_2 \exp(q\phi_2c-\frac{1}{2\ol\gs^2}c^2)dc.
	\end{align*}  
	The last estimate on the right hand side is finite for each $q\in\bR$ which proves the claim for $a_-$. 
	To bound the expectation of $a_+$, we may proceed in the same way by noting that 
	\bee
	\|a_+\|_{L^q(\gO)}\le \|\ol a\|_{L^\infty(\cD)}+\bE\big(|\sup_{x\in\cD}\Phi(W(x))|^q\big)^{1/q}+ \ol P\le \|\ol a\|_{L^\infty(\cD)}+\phi_1\bE\big(\sup_{x\in\cD}\exp(q\phi_2|W(\cdot,x)|)\big)^{1/q}+\ol P
	\eee
	by Assumption~\ref{ass:EV2}~(ii).
	Analogously, the claim follows for $a_{N,\eps,-}, a_{N,\eps,+}$ with the same bounds from above as for $a_-,a_+$ respectively, because $$\overline\sigma_N^2:=\sup_{x\in\cD}\bE(W_N(x)^2)\le\sum_{i=1}^N\eta_i\le tr(Q).$$ 
\end{proof}

\begin{thm}\label{thm:exis2}
	Let Assumption~\ref{ass:EV2} hold and $N\in\bN$ and $\eps>0$ be fixed. There exist for any $\go\in\gO$ unique pathwise weak solutions $u(\go,\cdot,\cdot)\in L^2(\bT;V)$ to Problem~\eqref{eq:pde} and $u_{N,\eps}(\go,\cdot,\cdot)\in L^2(\bT;V)$ to Problem~\eqref{eq:pde_approx}, respectively.  
	Moreover, the mappings $u, u_{N,\eps}:\gO\to L^2(\bT;V)$ are strongly measurable and satisfy
	for any $r\in [1,p)$ the estimates
	\bee
	\bE\Big(\sup_{t\in\bT} \|u\|^r_{*,t}\Big)^{1/r},\;\;\bE\Big(\sup_{t\in\bT} \|u_{N,\eps}\|^r_{*,t}\Big)^{1/r} 
	\le C\Big(\|u_0\|_{L^p(\gO;H)}+\|f\|_{L^p(\gO;L^2(\bT;V'))}\Big),	
	\eee
	where $C=C(r,a,b,T)>0$ is independent of $N$ and $\eps$.
\end{thm}
\begin{proof} To apply Theorem~\ref{thm:exis}, we need to verify Assumption~\ref{ass:para}.
	By Definition~\ref{def:a2}, Remark~\ref{rem:advection} and Lemma~\ref{lem:a2}, we have already covered Assumption~\ref{ass:para}(i),~(ii) and~(iv) for $a, b$ and $a_{N,\eps}, b_{N,\eps}$.
	From Lemma~\ref{lem:a2} we further obtain $1/a_-,1/a_{N,\eps,-}\in L^q(\gO;\bR)$ for any $q\in[1,\infty)$ and that $\|1/a_{N,\eps,-}\|_{L^q(\gO;\bR)}$ is bounded uniformly with respect to $N$ and $\eps$.
	For given $r\in[1,p)$, we then choose $q=(1/r-1/p)^{-1}<+\infty$ and the claim follows by Assumption~\ref{ass:EV2}(iii) and Theorem~\ref{thm:exis}.
\end{proof}

Having shown the existence and uniqueness of the weak solutions $u$ and $u_{N,\eps}$, we may bound the difference between both solutions in the (expected) parabolic norm with respect to the parameters $N$ and $\eps$. 
For this, we record the following estimate on the approximation error $a-a_{N,\eps}$.

\begin{thm}\cite[Theorem 3.12]{BS18b}\label{thm:a_error2}
	Under Assumption~\ref{ass:EV2}, it holds that
	\bee
	\bE(\|a-a_{N,\eps}\|^s_{L^\infty(\cD)})^{1/s}\le C\left(\Xi_N^{1/2}+\eps^{1/s}\right),
	\eee
	where $\Xi_N:=\sum_{i>N}\eta_i$ and $C>0$ is independent of $N\in\bN$ and $\eps>0$.	
\end{thm}

The final result of this section shows $u_{N,\eps}\to u$ in $L^r(\gO;L^2(T;V))$ as $N\to+\infty$ and $\eps\to0$.
\begin{thm}\label{thm:u_error2}
	Under Assumption~\ref{ass:EV2}, for any $r\in[1,(1/s+1/p)^{-1})$, the approximation error of $u$ is bounded in the parabolic norm by 
	\bee
	\bE\Big(\sup_{t\in\bT} \|u-u_{N,\eps}\|^r_{*,t}\Big)^{1/r} \le C\left(\Xi_N^{1/2}+\eps^{1/s}\right).	
	\eee
\end{thm}
\begin{proof}
	By Theorem~\ref{thm:exis2}, pathwise existence of solutions $u$ and $u_{N,\eps}$ to the variational Problems~\eqref{eq:var2}, \eqref{eq:var2_approx} is guaranteed, hence for all $\go\in\gO, t\in\bT$ and $v\in V$
	\begin{align*}
	\dualpair{V'}{V}{\partial_t u(\go,\cdot,t)}{v}+B_\go(u(\go,\cdot,t),v)
	=\dualpair{V'}{V}{\partial_t u_{N,\eps}(\go,\cdot,t)}{v}+B_\go^{N,\eps}(u_{N,\eps}(\go,\cdot,t),v).
	\end{align*}
	This may be reformulated as the variational problem to find $u-u_{N,\eps}\in L^2(\bT;V)$ such that for all $t\in\bT$ and $v\in V$
	\begin{align*}
	\dualpair{V'}{V}{\partial_t (u(\go,\cdot,t)&-u_{N,\eps}(\go,\cdot,t))}{v}+B_\go(u(\go,\cdot,t)-u_{N,\eps}(\go,\cdot,t),v)\\
	&=((a_{N,\eps}-a)(\go,\cdot),\nabla u_{N,\eps}(\go,\cdot,t)\cdot\nabla v)+((b_{N,\eps}-b)(\go,\cdot)\cdot\nabla  u_{N,\eps}(\go,\cdot,t),v)\\
	&=:\dualpair{V'}{V}{\widehat f(\go,\cdot,t)}{v},
	\end{align*}
	with initial value $(u-u_{N,\eps})(\go,\cdot,0)\equiv0$.
	Definition~\ref{def:a2} and Remark~\ref{rem:advection} imply
	\bee
	\|\widehat f(\go,\cdot,\cdot)\|_{L^2(\bT;V')}\le (1+\ol b_1)\|(a-a_{N,\eps})(\go,\cdot)\|_{L^\infty(\cD)}\|\sum_{i=1}^d |\partial_{x_i}u_{N,\eps}(\go,\cdot,\cdot)|\|_{L^2(\bT;H)},
	\eee
	and by Ineq.~\eqref{eq:nabla_norm} and Theorem~\ref{thm:exis2} we know that for $\ol r\in[1,p)$ 
	\begin{align*}
	\|\sum_{i=1}^d |\partial_{x_i}u_{N,\eps}|\|_{L^{\ol r}(\gO;L^2(\bT;H))}
	&\le 2^{d/2-1/2}\bE\Big(\big(\int_0^T|u|_{H^1(\cD)}^2dt\big)^{\ol r/2}\Big)^{1/\ol r}\\
	&\le 2^{d/2-1/2}\bE\Big(\|u_{N,\eps}\|_{*,T}^{\ol r}\Big)^{1/\ol r}\\
	&\le C\Big(\|u_0\|_{L^p(\gO;H)}+\|f\|_{L^p(\gO;L^2(\bT;V'))}\Big)<+\infty.
	\end{align*}
	We may now choose $\ol p\in[1,(1/s+1/\ol r)^{-1}]$ and obtain by H\"older's inequality and Theorem~\ref{thm:a_error2}
	\begin{align*}
	\|\widehat f(\go,\cdot,\cdot)\|_{L^{\ol p}(\gO;L^2(\bT;V'))}
	&\le C\bE(\|a-a_{N,\eps}\|^s_{L^\infty(\cD)})^{1/s}\|\sum_{i=1}^d \partial_{x_i} u_{N,\eps}\|_{L^{\ol r}(\gO;L^2(\bT;H))}
	\le C\left(\Xi_N^{1/2}+\eps^{1/s}\right)
	\end{align*}
	for some $C>0$ independent of $N$ and $\eps$. 
	The claim now follows with Lemma~\ref{lem:a2} and by applying Theorem~\ref{thm:exis} on $u-u_{N,\eps}$ for $q=(1/r-1/s-1/\ol p)^{-1}< (1/r-1/s-1/p)^{-1}<+\infty$.
\end{proof}

To draw samples of $u_{N,\eps}$, we need to employ further numerical techniques since $u_{N,\eps}(\go,\cdot,\cdot)$ is an element of the infinite-dimensional Hilbert space $L^2(\bT;V)$. 
Hence, we have to find pathwise approximations of $u_{N,\eps}$ in finite-dimensional subspaces of $L^2(\bT;V)$ by discretizing the spatial and temporal domain.  
Next, we construct suitable approximation spaces of $V$, combine them with a time stepping method and control for the discretization error.

\section{Pathwise discretization schemes}\label{sec:fem2}
In the previous section we demonstrated that $u$ may be approximated by $u_{N,\eps}$ for sufficiently large $N\in\bN$ resp. small $\eps>0$. Nevertheless, even $u_{N,\eps}(\go,\cdot,\cdot)$ will in general not be accessible analytically for fixed $\go, N$ and $\eps$, thus we need to find pathwise finite-dimensional approximations of $u_{N,\eps}(\go,\cdot,\cdot)$. In the first part of this section we explain how a semi-discrete solution may be obtained by approximating $V$ with a sequence of \textit{sample-adapted} Finite Element (FE) spaces. By sample-adaptedness we mean that the FE mesh is aligned \textit{a-priori} with the discontinuities of $P$ in each sample, i.e. the grid changes with each $\go\in\gO$. This is in contrast to \textit{adaptive} FE schemes based on a-posteriori error estimates that may require several stages of remeshing in each sample, see e.g. \cite{DDS19, EMN16, KY18}.
We analyze the discretization error for the pathwise sample-adapted strategy and further emphasize its advantages compared to a standard, sample-independent FE basis. 
In the second part we combine the spatial discretization with a backward time stepping scheme in $\bT$, with the time step chosen accordingly to the sample-dependent FE basis. 
Finally, we derive the mean-square error between the unbiased solution $u$ and the fully discrete approximation of $u_{N,\eps}$.

\subsection{Sample-adapted spatial discretization}
To find approximations of $u_{N,\eps}(\go,\cdot,t)\in V$ for fixed $\go\in\gO$ and $t\in\bT$, we  use a standard Galerkin approach based on a sequence $\cV_\go=(V_\ell(\go),\ell\in\bN_0)$ of finite-dimensional and sample-dependent subspaces $V_\ell(\go)\subset V$.
An obvious choice for $V_\ell$ is the space of piecewise linear FE with respect to some triangulation of $\cD$.
We follow the same approach as in \cite{BS18b} and utilize path-dependent meshes to match the interfaces generated by the jump-diffusion and -advection coefficients:
For a given random partition $\cT(\go)=(\cT_i,i=1\dots,\tau(\go))$ of $\cD$, we choose a triangulation $\cK_{\ell}(\go)$ of $\cD$ such that 
\bee
\cT(\go)\subset\cK_{\ell}(\go)\quad\text{and}\quad h_{\ell}(\go) :=\max_{K\in\cK_{\ell}(\go)}\text{diam}(K)\le \overline h_\ell\quad\text{for $\ell\in\bN_0$.}
\eee
Above,  $\text{diam}(K)$ is the longest side length of the triangle $K$ and $(\overline h_\ell,\ell\in\bN_0)$ is a positive sequence of deterministic refinement thresholds, decreasing monotonically to zero. This guarantees that $h_\ell(\go)\to0$ almost surely, although the absolute speed of convergence may vary for each $\go$. 
Given that the $\cT$ splits the domain $\cD$ into a finite number of piecewise linear polygons (see Assumption~\ref{ass:EV22} below), such a triangulation $\cK_\ell$ with $\cT(\go)\subset\cK_{\ell}(\go)$ always exists for any prescribed refinement $\ol h_\ell>0$.
Consequently, $V_\ell(\go)$ is chosen as the space of continuous, piecewise linear functions with respect to $K_\ell(\go)$, i.e. 
\be\label{eq:FE_space}
V_\ell(\go):=\{v_{\ell,\go}\in C^0(\ol\cD)\big|\;v_{\ell,\go}|_{\partial\cD}=0\;\text{and}\;v_{\ell,\go}|_K\in \cP_1(K),\, K\in \cK_\ell(\go)\}\subset V.
\ee
The set $\cP_1(K)$ denotes the space of all linear polynomials on the triangle $K$, and $\{v_{1,\go},\dots,v_{d_\ell(\go),\go}\}$ is the nodal basis of $V_\ell(\go)$ that corresponds to the vertices in $\cK_\ell(\omega)$.
As discussed in \cite[Section 4]{BS18b}, the adjustment of $\cK_\ell$ to the discontinuities of $a$ and $b$ accelerates convergence of the spatial discretization compared to a fixed, non-adapted FE approach. 
For a fixed triangle $K\in\cK_\ell(\go)$ let $x_1^K,x_2^K,x_3^K\in\ol\cD$ denote the corner points of $K$ and let $v_1^K, v_2^K, v_3^K\in\cP_1(K)$ be the corresponding linear nodal basis.
We define the \textit{local interpolation operator} on $K$ by
\bee
\cI(K):C^0(\ol K)\to \cP_1(K),\quad v\mapsto \sum_{i=1}^3 v(x_i^K)v_{i,\go}^K. 
\eee 
The \textit{global interpolation operator} $\cI_\ell:C^0(\ol\cD)\to V_\ell(\go)$ with respect to $\cK_\ell(\go)$ is then given by restrictions to the local operators, that is
\bee
[\cI_\ell v](x):=[\cI(K)v](x),\quad\text{for $x\in\cD$ and $K\in\cK_\ell$ is such that $x\in\ol K$.} 
\eee 
For simplicity, we only consider the nodal interpolation of continuous functions $v\in C^0(\ol\cD)$.

The semi-discrete version of Problem~\eqref{eq:var2_approx} is then to find $u_{N,\eps,\ell}(\go,\cdot,\cdot)\in L^2(\bT;V_\ell(\go))$ with $\partial_t u_{N,\eps,\ell}(\go,\cdot,\cdot)\in L^2(\bT;(V_\ell(\go))')$ such that for $t\in\bT$ and all $v_{\ell,\go}\in V_\ell(\go)$
\be\label{eq:semi_var}
\begin{split}
	\dualpair{V'}{V}{\partial_t u_{N,\eps,\ell}(\go,\cdot,t)}{v_{\ell,\go}}
	+B^{N,\eps}_\go(u_{N,\eps,\ell}(\go,\cdot,t),v_{\ell,\go})
	&=F_{t,\go}(v_{\ell,\go}),\\
	u_{N,\eps,\ell}(\go,\cdot,0)&=\cI_\ell u_0( \go,\cdot).
\end{split}
\ee
We have used the nodal interpolation $\cI_\ell u_0$ as approximation of the initial value, which is well-defined if $u_0(\go,\cdot)\in C^0(\ol\cD)$ holds for any $\go$ (see also Assumption~\ref{ass:EV22}(iii)/Remark~\ref{rem:EV_decay}).
The function $u_{N,\eps,\ell}(\go,\cdot,t)$ may be expanded with respect to the basis $\{v_{1,\go},\dots,v_{d_\ell(\go),\go}\}$ as
\be\label{eq:u_l_exp}
u_{N,\eps,\ell}(\go,x,t)=\sum_{j=1}^{d_\ell(\go)}c_j(\go,t)v_{j,\go}(x),
\ee
where the coefficients $c_1(\go,t),\dots,c_{d_\ell(\omega)}(\go,t)\in\bR$ depend on $(\go,t)\in\gO\times\bT$ and the respective coefficient column-vector is defined as
${\bf c(\go,t)} := (c_1(\go,t),\dots,c_{d_\ell(\omega)}(\go,t))^T$. With this, the semi-discrete variational problem in the finite-dimensional space $V_\ell(\go)$ is equivalent to solving the system of ordinary differential equations
\bee
\frac{d}{dt}\bf c(\go,t)+ \mathbf A(\go) {\bf c(\go,t)}
=\mathbf F(\go,t),\quad t\in\bT,
\eee
for $\bf c$ with stochastic stiffness matrix $(\mathbf A(\go))_{jk}=B^{N,\eps}_\go(v_{j,\go},v_{k,\go})$ and time-dependent load vector $(\mathbf F(\go,t))_j=F_{t,\go}(v_{j,\go})$ for $j,k\in\{1,\dots,d_\ell(\go)\}$.
To ensure the well-posedness of Eq.~\eqref{eq:semi_var} and derive error bounds of the numerical approximation of $u$ in a mean-square sense, we need to modify Assumption~\ref{ass:EV2} ((ii) and (iv) are unaltered):

\begin{assumption}\label{ass:EV22}
	~
	\begin{enumerate}[label=(\roman*)]
		\item The eigenfunctions $e_i$ of $Q$ are continuously differentiable on $\cD$ and there exist constants $\ga,\gb,C_e,C_\eta>0$ such that $2\ga\le\beta$ and for any $i\in\bN$
		\bee
		\|e_i\|_{L^\infty(\cD)}\le C_e,\quad\max_{j=1,\dots,d}\|\partial_{x_j} e_i\|_{L^\infty(\cD)}\le C_ei^\alpha\quad\text{and}\quad\sum_{i=1}^\infty\eta_ii^\gb\leq C_\eta<+\infty.
		\eee 
		\item Furthermore, the mapping $\Phi$ as in Definition~\ref{def:a2} and its derivative are bounded for $w\in\bR$ by
		\bee
		\phi_1\exp(\phi_2 w)\ge \Phi(w)\ge\phi_1\exp(-\phi_2 w),\quad|\frac{d}{dx}\Phi(w)|\le\phi_3\exp(\phi_4|w|),
		\eee
		where $\phi_1,\dots,\phi_4>0$ are arbitrary constants.
		\item There exists $p>2$ such that $f,\partial_t f\in L^p(\gO;L^2(\bT;H))$ and $u_0\in L^p(\gO;V)\cap L^p(\gO;H^{1+\epsilon}(\cD))$ for some arbitrary $\epsilon>0$. Furthermore, $u_0$ and $f$ are stochastically independent of $\cT$. 
		\item The sequence $(P_i, i\in\bN)$ consists of nonnegative and bounded random variables $P_i\in [0,\ol P]$ for some $\ol P>0$. In addition, for $s>2$ such that $1/p+1/s<1/2$ there exists a sequence of approximations $(\widetilde P_i, i\in\bN)\subset[0,\ol P]^\bN$ so that the sampling error is bounded, for some $\eps>0$, by
		\begin{equation*}
		\bE(|\widetilde P_i-P_i|^s)\le\eps,\quad i\in\bN.
		\end{equation*}
		\item The partition elements $\cT_i(\go)$ are \textit{polygons with piecewise linear boundary} and a finite number of boundary edges for all $\go\in\gO$ and $\bE(\tau^n)<\infty$ for any $n\in\bN$.
		\item Let $2^V$ be the power set of $V$. For all $\ell\in\bN_0$, the correspondence $\gO\to 2^V,\;\go\mapsto V_\ell(\go)$ admits non-empty values and is \textit{weakly measurable}: for each open subset $\widetilde V\subset V$ it holds that
		\bee
		\{\go\in\gO|\;V_\ell(\go)\cap \widetilde V\neq\emptyset\}\in\cF.
		\eee 
		\item \textit{Conformity}: In dimension $d=2$, let $K_1,K_2\in\cK_\ell(\go)$ for some fixed $\ell\in\bN_0$ and $\go\in\gO$. Then, the intersection $\ol K_1\cap \ol K_2$ is either empty, a common edge or a common vertex of $\cK_\ell(\go)$.
		\item \textit{Shape-regularity}: Let $\rho_{K,out}$ and $\rho_{K,in}$ denote the radius of the outer respectively inner circle of the triangle $K$. Then, there is a constant $\ol\rho>0$ such that 
		\bee
		\esssup_{\go\in\gO}\;\sup_{\ell\in\bN_0}\;\sup_{K\in\cK_{\ell}(\go)}\frac{\rho_{K,out}}{\rho_{K,in}}\le \ol\rho<+\infty.
		\eee	
	\end{enumerate}
\end{assumption}

\begin{rem}\label{rem:EV_decay}
	We discuss Assumption~\ref{ass:EV22} in the following:
	\begin{itemize}
		\item Assumption~\ref{ass:EV22}(i) implies for all $i=1,\dots,d$ and $x\in\cD$ 
		\bee
		\bE(|\partial_{x_i} W_N(x)|^2)=\bE(|\sum_{j=1}^n \sqrt{\eta_j}\partial_{x_i} e_j(x)_j Z_j|^2)\le C_e\sum_{j=1}^N\eta_j j^{2\ga}\le C_e\sum_{j=1}^N\eta_j j^{\gb},
		\eee
		hence there exist an $L^2(\gO;\bR)$-limit $\partial_{x_i}W(\cdot,x):=\lim_{N\to+\infty}\partial_{x_i} W_N(\cdot,x)$. 
		Hence, $2\ga\le\gb$ entails the mean-square differentiability (or pathwise Lipschitz-continuity) of the Gaussian field $W$. 
		\item By the fractional Sobolev inequality (\cite[Theorem 6.7]{DGV12}), $u_0(\go,\cdot)\in H^{1+\epsilon}(\cD)$ for $\epsilon>0$ implies with $d\le 2$ that $u_0(\go,\cdot)\in C^0(\ol\cD)$ and the nodal interpolation of $u_0$ is well-defined. The assumptions on $f$ and $\partial_tf$ are necessary to control the error of a temporal discretization scheme. 
		The nodal basis functions $v_{j,\go}$ are solely determined by $\cT(\go)$ and since $f, u_0$ are stochastically independent of $\cT$, we may expand the sample-adapted semi-discrete solution via Eq.~\eqref{eq:u_l_exp}, i.e. obtain a separation of spatial and temporal variables.
		\item The condition $1/p+1/s<1/2$ enables us to derive all errors in a mean-square sense. Furthermore, the partition into piecewise linear polygons enables us to construct triangulations $\cK_\ell(\go)$ resp. approximation spaces $V_\ell(\go)$ as in Eq.~\eqref{eq:FE_space}
		\item The weak measurability of the correspondence $\go\mapsto V_\ell(\go)$ ensures the (strong) measurability of the approximated solution $u_{N,\eps,\ell}:\gO\to L^2(\bT;V)$, see Proposition~\ref{prop:meas2}. 
		This assumption is necessary, since pathological approximation spaces $V_\ell(\go)$ may still be constructed on a nullset of $\gO$, even under Assumption~\ref{ass:EV22}(iv).
		For details on measurable correspondences we refer to \cite[Chapter 18]{AB06}.
		\item Conformity and shape-regularity of the FE triangulations are necessary to control the FE discretization error.
	\end{itemize}
\end{rem}

We show measurability of the semi-discrete approximations and record a bound on the interpolation error. 

\begin{prop}\label{prop:meas2}
	Let Assumption~\ref{ass:EV22} hold and let $\ell\in\bN_0$ be fixed. Then, for any $\go\in\gO$ there exists a unique sample-adapted solution $u_{N,\eps,\ell}(\go,\cdot,\cdot)\in L^2(\bT;V)$ to the semi-discrete problem~\eqref{eq:semi_var} and the mapping $u_{N,\eps,\ell}:\gO\to L^2(\bT;V)$ is strongly measurable. 
\end{prop}

\begin{proof}
	For fixed $\go$, existence and uniqueness of $u_{N,\eps,\ell}(\go,\cdot,\cdot)$ follows with Assumption~\ref{ass:EV22} as in Theorem~\ref{thm:exis}, hence the map $u_{N,\eps,\ell}:\gO\to L^2(\bT; V)$ is well-defined.
	To show measurability, we use again the space $\cX:=L^2(\bT;V)\times L^2(\bT;V')$ with $\|(y_1,y_2)\|_\cX:=\|y_1\|_{L^2(\bT;V)}+\|y_2\|_{L^2(\bT;V')}$ as in the proof of Theorem~\ref{thm:exis}. 
	Let $\{v_{1,\go},\dots,v_{d_\ell(\go),\go}\}$ be a basis of $V_\ell(\go)$ and define the sequence 
	\bee
	\widetilde v_{i,\go}:=\begin{cases} v_{i,\go}\quad&\text{if $i\le d_\ell(\go)$}\\
		v_{d_\ell(\go),\go}\quad&\text{if $i> d_\ell(\go)$}\end{cases}.
	\eee
	By Assumption~\ref{ass:EV22}(vi), the correspondence $\go\mapsto V_\ell(\go)$ is weakly measurable and has closed, non-empty values, therefore there exists a sequence $(\xi_i,i\in\bN)$ of measurable functions $\xi_i:\gO\to V$ such that $\xi_i(\go)\in V_\ell(\go)$ and $V_\ell(\go)=\overline{\{\xi_1(\go),\xi_2(\go),\dots\}}$ (see \cite[Corollary 18.14]{AB06}). Consequently, each $\widetilde v_{i,\cdot}:\gO\to V$ can be written as the limit of measurable functions and is therefore $\cF-\cB(V)$-measurable.    
	Now, consider the functional 
	\begin{align*}
	\widetilde J^{N,\eps}_i:\gO\times \cX\to\bR,\quad (\go,w)\mapsto &\int_0^TB^{N,\eps}_\go(w(\cdot,t),\widetilde v_{i,\go})-F_{\go,t}(\widetilde v_{i,\go})+\dualpair{V'}{V}{\partial_t w(\cdot,t)}{\widetilde v_{i,\go}}\\
	&\quad+\|w(\cdot,t)-\sum_{i=1}^{d_\ell(\go)}(w(\cdot,t),\widetilde v_{i,\go})_V\widetilde v_{i,\go}\|_Vdt
	\end{align*}
	By Theorem~\ref{thm:exis} and the measurability of $\widetilde v_{i,\cdot}$ we conclude that $J^{N,\eps}_i$ is a Carath\'eodory mapping. 
	We define the correspondence 
	\bee
	\widetilde \varphi_i(\go):=\{w\in \cX|\, J^{N,\eps}_i(\go,w)=0\}.
	\eee
	and obtain again by \cite[Corollary 18.8]{AB06} that the graph $\text{Gr}(\widetilde\varphi_i)=\{(\go,w)\in\gO\times \cX|\, w\in\varphi_i(\go)\}$ is measurable.
	By construction of $J^{N,\eps}_i$, we get
	\bee
	\{(\go,u_{N,\eps,\ell}(\go,\cdot,\cdot),\partial_tu_{N,\eps,\ell}(\go,\cdot,\cdot))|\,\go\in\gO\}=\bigcap_{i\in\bN}\text{Gr}(\widetilde \varphi_i)\in\cF\otimes\cB(\cX),
	\eee 
	and the claimed measurability of $u_{N,\eps,\ell}$ follows analogously as in the proof of Theorem~\ref{thm:exis}.
\end{proof}

\begin{lem}\label{lem:interpolation}
	Under Assumption~\ref{ass:EV22}, let $\go\in\gO$ be fixed, and let $v\in H^{\vartheta}(\cT_i)$ for some $\vartheta\in (1,2]$ and $i=1,\dots,\tau(\go)$. Then, $\cI_\ell v\in C^0(\ol\cT_i)$ is well-defined on each partition element $\cT_i$ and for $m\in\{0,1\}$ there holds 
	\be\label{eq:int_err} 
	\big(\sum_{i=1}^{\tau(\go)}\|(1-\cI_\ell)v\|^2_{H^m(\cT_i)}\big)^{1/2}
	=\big(\sum_{i=1}^{\tau(\go)}\sum_{K\in\cT_i}\|(1-\cI(K))v\|^2_{H^m(K)}\big)^{1/2}
	\le C\ol h_\ell^{\vartheta-m}\big(\sum_{i=1}^{\tau(\go)}|v|_{H^{\vartheta}(\cT_i)}\big)^{1/2},
	\ee
	where $C=C(\ol\rho,\vartheta,m,d)>0$ is a deterministic constant.
\end{lem}
\begin{proof}
	By the Sobolev embedding theorem $\|v\|_{C^0(\ol{\cT_i})}\le C \|v\|_{H^{\vartheta}(\cT_i)}$ and thus $\cI_\ell v$ is well-defined on $\cT_i$.
	Moreover, for $m=\{0,1\}$, we use that $V_\ell(\go)\neq\emptyset$ and the interpolation estimates from \cite[Theorem 4.4.20]{BS07} to see that
	\bee
	\|\cI_\ell v\|_{H^m(\cT_i)}\le C \|v\|_{H^m(\cT_i)}
	\eee
	for a constant $C=C(\ol\rho,m,d)>0$, independent of $\cT_i$.
	Together with $w=\cI_\ell w$ for any $w\in V_\ell(\go)$ we then obtain
	\begin{align*}
	\|(1-\cI_\ell)v\|_{H^m(\cT_i)}&\le \inf_{w\in V_\ell(\go)} \|v-w\|_{H^m(\cT_i)} + \|\cI_\ell(w-v)\|_{H^m(\cT_i)}\\
	&\le C \inf_{w\in V_\ell(\go)}\|v-w\|_{H^m(\cT_i)}\\
	&\le C \inf_{w\in V_\ell(\go)} \big(\sum_{K\in\cT_i}\|v-w\|^2_{H^m(K)}\big)^{1/2}.
	\end{align*}	
	Assumption~\ref{ass:EV22} guarantees that $K\in\cT_i$ holds by construction of the approximation space $V_\ell(\go)$.
	The claim now follows for instance by the estimates from \cite[Chapter 8.5]{H10} and by the fact that the constant $C=C(\ol\rho,\vartheta,m,d)>0$ is deterministic, i.e. independent of $\cT_i$.
\end{proof}

To bound the pathwise FE discretization error, we now fix $\go\in\gO, t\in\bT$ and $u_{N,\eps}(\go,\cdot,t)\in V$ and consider the corresponding pathwise elliptic PDE
\be\label{eq:elliptic}
\begin{split}
	-\nabla\cdot (a_{N,\eps}(\go,\cdot)\nabla u_{N,\eps}(\go,\cdot,t))
	&=f(\go,\cdot,t)-b_{N,\eps}(\go,\cdot)\cdot\nabla u_{N,\eps}(\go,\cdot,t)-\partial_t u_{N,\eps}(\go,\cdot,t)
	=:\widetilde f(\go,\cdot,t)
\end{split}
\ee
on $\cD$ with homogeneous Dirichlet boundary conditions.
Let $\cE$ be the set of all interior edges of $\cT(\go)$ and for every $e\in\cE$ let $i_e,i'_e\in\{1,\dots,\tau(\go)\}$ with $i_e\neq i'_e$ be the indices such that $e=\ol\cT_{i_e}\cap\ol\cT_{i'_e}$.
Accordingly, the outward normal vectors on either side of $e$ with respect $\cT_{i_e}$ and $\cT_{i'_e}$ are denoted by $\vv n_{i_e}$ and $\vv n_{i'_e}$, respectively.
Due to the discontinuities of $a_{N,\eps}(\go,\cdot)$, this yields the transition condition
\be\label{eq:interface}
a_{N,\eps}(\go,\cdot)\vv n_{i_e}\cdot \nabla u_{N,\eps}(\go,\cdot,t)=a_{N,\eps}(\go,\cdot)\vv n_{i'_e}\cdot \nabla u_{N,\eps}(\go,\cdot,t)\quad\text{on $e\in\cE$}.
\ee
Therefore, $u_{N,\eps}(\go,\cdot,t)$ may be regarded as weak solution to an \textit{elliptic interface problem} given by Eqs.~\eqref{eq:elliptic}-\eqref{eq:interface} satisfying for all $v\in V$
\bee
\int_\cD a_{N,\eps}(\go,x)\nabla u_{N,\eps}(\go,x,t)\cdot\nabla v(x)dx\\
=\int_\cD \widetilde f(\go,x,t)v(x)dx.
\eee
Given that $\widetilde f(\go,\cdot,t)\in H$ (which is verified almost surely by Lemmas~\ref{lem:w1infa} and~\ref{lem:dtu_H} below), it is known for dimension $d=2$ (e.g. from \cite{N90, NS94a, NS94b, P01}) that the solution $u_{N,\eps}(\go,\cdot,t)$ to the elliptic interface problem admits a \textit{decomposition into singular functions} with respect to the corners of $\cT(\omega)$.
More precisely, 
\be\label{eq:sing_decom}
u_{N,\eps}(\go,\cdot,t)=w+\sum_{j\in\cS}c_j\chi_j(r^{(j)})\psi_j(r^{(j)},\varphi^{(j)}),
\ee
where $\cS$ denotes the set of \textit{singular points} in the partition $\cT(\omega)$ (in our case $\cS$ is the set of corners in $\cT(\omega)$) and $(r^{(j)},\varphi^{(j)})$ are polar coordinates with respect to the singular point $j\in\cS$. For any $i=1,\dots,\tau$, it holds $w\in H^2(\cT_i)$ and $\psi\not\in H^{1+\gk_j}(\cT_i)$, but $\psi\in H^{1+\gk_j-\epsilon}(\cT_i)$ for some $\gk_j\in (0,1]$ and any $\epsilon>0$. Moreover, $c_j\in\bR$ are coefficients and $\chi_j$ is a smooth and bounded cutoff function vanishing near the singular point $j$. 

The decomposition in Eq.~\eqref{eq:sing_decom} shows that $u_{N,\eps}(\go,\cdot,t)\not\in H^2(\cT_i)$, but we may expect a piecewise regularity of $u_{N,\eps}(\go,\cdot,t)\in H^{1+\ul\gk-\eps}(\cT_i)$, where $\ul\gk:=\min_{j\in\cS}\gk_j$. The precise values of the exponents $\gk_j\in (0,1]$ depend on the shape of the partition elements $\cT_i$, i.e. their angle at the singular points $\cS$, as well as on the magnitude of the jump heights $P_i$. 
Furthermore, the results from \cite{N90,NS94a} show that the coefficients $c_j$ and $w$ depend continuously of the  on the right hand side $\widetilde f$, the gradient of $a_{N,\eps}$ on $\cT_i$ and the inverse of $a_{N,\eps,-}$.
A detailed analysis on the dependencies of $\gk_j$, $c_j$ and $w$ may be found in the literature (see~\cite{NS94a, NS94b}). 
For sake of simplicity we assume piecewise regularity of $u_{N,\eps}$ motivated by the decomposition in Eq.~\eqref{eq:sing_decom}.

\begin{assumption}\label{ass:fe}
	There are deterministic constants $\gk\in (0,1]$ and $C>0$, such that for all $N\in\bN, \eps>0$ and $t\in\bT$ there holds for almost all $\omega\in\Omega$ and for $i=1,\dots,\tau(\go)$,
	\bee
	\|u_{N,\eps}(\go,\cdot,t)\|_{H^{1+\gk}(\cT_i)} \le C\frac{\|\widetilde f(\go,\cdot,t)\|_{L^2(\cT_i)}+\|a_{N,\eps}(\go,\cdot)\|_{\cW^{1,\infty}(\cT_i)}\|u_{N,\eps}(\go,\cdot,t)\|_{H^1(\cT_i)}}{a_{N,\eps,-}(\go)}.
	\eee
\end{assumption}

\begin{rem}\label{rem:fe}
	Assumption~\ref{ass:fe} is made to simplify the following numerical analysis for $d=2$, whereas this assumption would not be necessary in $d=1$. 
	It is based on the decomposition in Eq.~\eqref{eq:sing_decom} as well as the estimates for $c_j$ and $w$ from \cite{N90} in terms of the right hand side $\widetilde f$ and $a_{N,\eps}$.
	Although it may seem artificial at a first glance, we recover $\gk$ close to one in the numerical examples from Section~\ref{sec:num2} for $d=2$. 
	On the other hand, it is actually possible to obtain lower bounds on $\gk$, i.e. to ensure a certain minimum of piecewise regularity almost surely. This is for instance the case if:
	\begin{itemize}
		\item the jump heights $\cP_i$ and the maximum interior angles of the $\cT_i$ are bounded from above and below, or
		\item if $a_{N,\eps}$ satisfies almost surely a \textit{quasi-monotonicity condition}, 
	\end{itemize}
	see \cite{P01} and the references therein.
	Since $d\le2$ and $u_{N,\eps}(\go,\cdot,t)\in H^{1+\gk}(\cT_i)$ holds for every polygonal subdomain $\cT_i$, it follows that $u_{N,\eps}(\go,\cdot,t)\in H^\vartheta(\cD)$, where $\vartheta=\min(1+\gk,3/2-\epsilon)$ for any $\epsilon>0$ (see \cite[Lemma 3.1]{P01}). This in turn yields that $u_{N,\eps}(\go,\cdot,t)\in C^0(\ol\cD)$ by the fractional Sobolev inequality. Hence, the nodal interpolation $\cI_\ell u_{N,\eps}(\go,\cdot,t)$ is well-defined.
\end{rem}

We are now ready to state our main result on the spatial discretization error:

\begin{thm}\label{thm:semi_error_a}
	Let Assumptions~\ref{ass:EV22} and~\ref{ass:fe} hold and let $u_{N,\eps,\ell}$ be the sample-adapted FE approximation of $u_{N,\eps}$ as in Eq.~\eqref{eq:semi_var}, where $\ol h_\ell\le 1$.
	Then, there is a $C>0$, independent of $N,\eps$ and $\ol h_\ell$ such that   
	\bee
	\bE\big(\sup_{t\in\bT}\|u_{N,\eps}-u_{N,\eps,\ell}\|^2_{*, t}\big)^{1/2}\le C\ol h_\ell^\gk.
	\eee
\end{thm}

For the proof of Theorem~\ref{thm:semi_error_a} we record several technical lemmas as preparation.

\begin{lem}\label{lem:w1infa}
	Let $\gT\in\bR^d$ be an open, bounded domain and denote by $\cW^{k,\infty}(\gT)$ the Sobolev space defined by the (semi-)norm 
	\bee
	\|v\|_{\cW^{k,\infty}(\gT)}:=\sum_{|\nu|\le k}\|D^\nu v\|_{L^\infty(\gT)},\quad|v|_{\cW^{k,\infty}(\gT)}:=\sum_{|\nu|= k}\|D^\nu v\|_{L^\infty(\gT)},\quad k\in\bN, 
	\eee
	for any measurable mapping $v:\gT\to\bR$.
	Under Assumption~\ref{ass:EV22}, for any $q\in[1,\infty)$ 
	\bee
	\big\|\max_{i=1,\dots,\tau}\|a_{N,\eps}\|_{\cW^{1,\infty}(\cT_i)}\big\|_{L^q(\gO;\bR)}\le C<+\infty,
	\eee
	where $C=C(q)>0$ is independent of $N$ and $\eps$.
\end{lem}
\begin{proof}
	As $a_{N,\eps}$ is almost surely continuously differentiable on each partition element $\cT_i$ by Assumption~\ref{ass:EV22}, we have
	\bee
	\|a_{N,\eps}(\go,\cdot)\|_{\cW^{1,\infty}(\cT_i)}\le a_{N,\eps,+}(\go)+\max_{i=1,\dots,d}
	\|\partial_{x_i} \ol a\|_{L^\infty(\cD)}
	+\|\frac{d}{dx}\Phi(W_N(\go,\cdot))\partial_{x_i}W_N(\go,\cdot)\|_{L^\infty(\cD)}
	\eee
	with $\|\partial_{x_i}\ol a\|_{L^\infty(\cD)}<+\infty$ for all $i=1,\dots,d$.
	Moreover, Lemma~\ref{lem:a2} states that $\|a_{N,\eps,+}\|_{L^q(\gO;\bR)}<+\infty$ for any $q\in[1,\infty)$ and the norm is bounded uniformly with respect to $N$ and $\eps$.
	Thus, we only need to estimate the last term on the right hand side.  
	We use H\"older's inequality and Assumption~\ref{ass:EV22} to obtain for any $q\ge1$
	\begin{align*}
	\big\|\frac{d}{dx}\Phi(W_N)\partial_{x_i} W_N\big\|_{L^q(\gO;L^\infty(\cD))}&\le \big\|\frac{d}{dx}\Phi(W_N)\big\|_{L^{2q}(\gO;L^\infty(\cD))}\|\partial_{x_i} W_N\|_{L^{2q}(\gO;L^\infty(\cD))}\\
	&\le \phi_3 \bE(\exp(2q\phi_4||W_N||_{L^\infty(D)}))^{1/(2q)}||\partial_{x_i}W_N||_{L^{2q}(\gO;L^\infty(\cD))}.
	\end{align*}
	The random field $W_N$ is centered Gaussian with $\sup_{x\in\cD}\bE(W_N(x)^2)\le \sup_{x\in\cD}\bE(W(x)^2)\le tr(Q)$
	and we proceed as in Lemma~\ref{lem:a2} to conclude that  
	\begin{align*}
	\bE(\exp(2q\phi_4||W_N||_{L^\infty(D)}))\le\int_0^\infty 2q\phi_4\exp(2q\phi_2c)\bP(||W||_{L^\infty(\cD)}>c)dc<+\infty.
	\end{align*}
	
	To estimate $||\partial_{x_i} W_N||_{L^{2q}(\gO;L^\infty(\cD))}$, we note that, for $x\in\cD$, $\partial_{x_i} W_N(x)$ is also centered Gaussian with variance $\sum_{j=1}^N\eta_j (\partial_{x_i} e_j(x))^2$.
	For any $N\in\bN$
	\bee
	\sup_{x\in\cD} |\partial_{x_i} W_N(x)|
	=\sup_{x\in\cD}\left|\sum_{j=1}^N\sqrt{\eta_j}\partial_{x_i} e_j(x) Z_j\right| 
	\le \sum_{j=1}^N\sqrt{\eta_j}j^\ga |Z_j|
	\eee 
	by Assumption~\ref{ass:EV22}(i), hence $\partial_{x_i} W_N$ is almost surely bounded on $\cD$.
	The symmetric distribution of $\partial_{x_i} W_N(x)$ and \cite[Theorem 2.1.1]{AT09} then imply $\bE(\sup_{x\in\cD}\partial_{x_i} W_N(x))\ge0$, 
	\bee
	\bE(||\partial_{x_i} W_N||_{L^\infty(\cD)})\le 2\bE(\sup_{x\in\cD}\partial_{x_i} W_N(x))=:2E_{N,i}<+\infty,\quad\text{and}
	\eee
	\be\label{eq:tail21}
	\bP(\sup_{x\in\cD}\partial_{x_i} W_N(x)>c)\le \exp(-\frac{(c-E_{N,i})^2}{2\widetilde\sigma_{N,i}^2})\le \exp(-\frac{(c-E_{N,i})^2}{2\widetilde\sigma^2}),\quad c>0,
	\ee
	analogously to Lemma~\ref{lem:a2}. The maximal variances in Ineq.~\eqref{eq:tail21} are given by  
	\bee
	\widetilde \sigma_{N,i}^2:=\sup_{x\in\cD}\bE((\partial_{x_i} W_N(x))^2)=\sum_{j=1}^N\eta_j (\partial_{x_i}e_j(x)) \le \widetilde\sigma^2:=C_e\sum_{j=1}^\infty \eta_jj^{2\ga}\le C_e\sum_{j=1}^\infty \eta_jj^{\gb}<+\infty.	
	\eee
	Without loss of generality, we assume $q\in\bN$ to obtain $\bE(||\partial_{x_i} W_N||^{2q}_{L^\infty(\cD)})=\bE(\sup_{x\in\cD}(\partial_{x_i} W_N(x))^{2q})$.
	We now have to make sure that $\bE(\sup_{x\in\cD}(\partial_{x_i} W_N(x))^{2q})$ is bounded uniformly in $i$ and $N$.
	Similar to Lemma~\ref{lem:a2}, Fubini's Theorem and Ineq.~\eqref{eq:tail21} yield
	\be\label{eq:tail22}
	\begin{split}
		\bE(\sup_{x\in\cD}(\partial_{x_i} W_N(x))^{2q})&=\int_0^\infty \bP(\sup_{x\in\cD}(\partial_{x_i} W_N(x))^{2q}>c)dc\\
		&\le \int_{0}^\infty\exp(-\frac{(c^{1/(2q)}-E_{N,i})^2}{2\widetilde\sigma^2})dc\\
		&\le \int_\bR \exp(-\frac{|c|^{1/q}}{2\widetilde\sigma^2})dc,
	\end{split}
	\ee
	and the last integral is finite for any $q\in[1,\infty)$ and independent of $N$ and $i$.
\end{proof}

\begin{lem}\label{lem:dtu_H}
	Under Assumption~\ref{ass:EV22}, for any $r\in[1,p)$ it holds that 
	\bee
	\|\partial_t u_{N,\eps}\|_{L^r(\gO;L^2(\bT;H))}+\big\|\sup_{t\in\bT} \|u_{N,\eps}(\cdot,\cdot,t)\|_V\big\|_{L^r(\gO;\bR)}
	\le C\Big(\|u_0\|_{L^p(\gO;V)} +\|f\|_{L^p(\gO;L^2(\bT;H))}\Big)
	\eee
	as well as 
	\bee
	\|\partial_t u_{N,\eps,\ell}\|_{L^r(\gO;L^2(\bT;H))}+\big\|\sup_{t\in\bT} \|u_{N,\eps,\ell}(\cdot,\cdot,t)\|_V\big\|_{L^r(\gO;\bR)}
	\le C\Big(\|u_0\|_{L^p(\gO;H^{1+\epsilon}(\cD))} +\|f\|_{L^p(\gO;L^2(\bT;H))}\Big).
	\eee
	
\end{lem}
\begin{proof}
	We use the first part of the proof from \cite[Chapter 7.1, Theorem 5]{E10} to obtain the pathwise estimate 
	\begin{align*}
	\|\partial_t u_{N,\eps}(\go,\cdot,\cdot)\|^2_{L^2(\bT;H)}&+\sup_{t\in\bT}\int_\cD a_{N,\eps}(\go,x,t)\nabla u_{N,\eps}(\go,x,t)\cdot\nabla u_{N,\eps}(\go,x,t)dx\\
	\le &\int_\cD a_{N,\eps}(\go,x,t)\nabla u_{N,\eps}(\go,x,0)\cdot\nabla u_{N,\eps}(\go,x,0)dx\\
	&+\int_0^T\|b_{N,\eps}(\go,x,t)\cdot\nabla  u_{N,\eps}(\go,\cdot,t)\|^2_Hdt + \|f(\go,\cdot,\cdot)\|^2_{L^2(\bT;H)}\\
	\le &a_{N,\eps,+}(\go)\|u_0(\go,\cdot)\|_V^2 +\ol b^2_22^{d-1}\|u(\go,\cdot,\cdot)\|^2_{T,*}+\|f(\go,\cdot,\cdot)\|^2_{L^2(\bT;H)}.
	\end{align*}
	In the last step, we have used that $\|b_{N,\eps}(\go,x)\|_\infty\le \ol b_2$ (see Remark~\ref{rem:advection}) as well as Ineq.~\eqref{eq:nabla_norm}.
	On the other hand, we have the lower bound
	\begin{align*}
	\|\partial_t u_{N,\eps}(\go,\cdot,\cdot)\|^2_{L^2(\bT;H)}&+\sup_{t\in\bT}\int_\cD a_{N,\eps}(\go,x)\nabla u_{N,\eps}(\go,x,t)\cdot\nabla u_{N,\eps}(\go,x,t)dx\\
	\ge&\|\partial_t u_{N,\eps}(\go,\cdot,\cdot)\|^2_{L^2(\bT;H)}+a_{N,\eps,-}(\go)\sup_{t\in\bT} |u_{N,\eps}(\go,\cdot,t)|^2_{H^1(\cD)}.
	\end{align*}
	Since the norms $|\cdot|_{H^1(\cD)}$ and $\|\cdot\|_{H^1(\cD)}=\|\cdot\|_V$ are equivalent by the Poincar\'e inequality, we treat $a_{N,\eps,-}$ once more in the fashion of Theorem~\ref{thm:exis} to arrive at the estimate 
	\begin{align*}
	\|\partial_t u_{N,\eps}(\go,\cdot,\cdot)\|^2_{L^2(\bT;H)}&+\sup_{t\in\bT}\int_\cD \|u_{N,\eps}(\go,x,t)\|^2_Vdx\\
	\le &C(1+1/a_{N,\eps,-}(\go))\Big(a_{N,\eps,+}(\go)\|u_0(\go,\cdot)\|_V^2 +\|u(\go,\cdot,\cdot)\|^2_{T,*}+\|f(\go,\cdot,\cdot)\|^2_{L^2(\bT;H)}\Big).
	\end{align*}
	The claim now follows with $1/a_{N,\eps,-},a_{N,\eps,+}\in L^q(\gO;\bR)$ for arbitrary large $q\in[1,\infty)$, H\"older's inequality and Theorem~\ref{thm:exis}.
	The proof for the estimate on $u_{N,\eps,\ell}$ may be carried out analogously with the initial condition $u_{N,\eps,\ell}(\cdot,\cdot,0)=\cI_\ell u_0$ and by observing that with Lemma~\ref{lem:interpolation}
	\bee
	\|\cI_\ell u_0\|_{L^p(\gO;V)}\le \|\cI_\ell u_0-u_0\|_{L^p(\gO;V)}+\|u_0\|_{L^p(\gO;V)}\le C \|u_0\|_{L^p(\gO;H^{1+\epsilon}(\cD))}.
	\eee
	
\end{proof}

\begin{lem}\label{lem:regularity} 
	Under Assumption~\ref{ass:EV22} and~\ref{ass:fe}, for any $r\in[2,p)$ it holds that 
	\begin{align*}
	\bE\Big(\big(\int_0^T\sum_{i=1}^\tau \|u_{N,\eps}\|_{H^{1+\gk}(\cT_i)}^2dt\big)^{r/2}\Big)^{1/r}<+\infty.
	\end{align*}
\end{lem}
\begin{proof}
	Assumptions~\ref{ass:EV22}(iv) and~\ref{ass:fe} yield for fixed $\go$ and $t$
	\bee
	\begin{split} 
		\sum_{i=1}^{\tau(\go)} \|u_{N,\eps}(\go,\cdot,t)\|_{H^{1+\gk}(\cT_i)}^2
		&\le C \frac{\|\widetilde f(\go,\cdot,t)\|_H^2+\|u_{N,\eps}(\go,\cdot,t)\|_V^2\sum_{i=1}^{\tau(\go)}\|a_{N,\eps}(\go,\cdot)\|^2_{\cW^{1,\infty}(\cT_i)}}
		{a_{N,\eps,-}(\go)^2} \\
		&\le C \frac{\|\widetilde f(\go,\cdot,t)\|_H^2+\|u_{N,\eps}(\go,\cdot,t)\|_V^2\tau(\go)\max_{i=1,\dots,\tau(\go)}\|a_{N,\eps}(\go,\cdot)\|^2_{\cW^{1,\infty}(\cT_i)}}
		{a_{N,\eps,-}(\go)^2}. 
	\end{split}
	\eee
	Now, we integrate with respect to $\bT$ and $\gO$, and use H\"older's inequality to obtain for $r\in[2,p)$
	\begin{align*}
	\bE\Big(\big(\int_0^T\sum_{i=1}^\tau \|u_{N,\eps}\|_{H^{1+\gk}(\cT_i)}^2dt\big)^{r/2}\Big)^{1/r}
	&\le C\Big(\|1/a_{N,\eps,-}\|_{L^{q}(\gO;L^2(\bT;H)}\|\widetilde f\|_{L^{r_1}(\gO;L^2(\bT;H))}\\
	&\qquad+\|1/a_{N,\eps,-}\|_{L^{4q}(\gO;L^2(\bT;H)}\|\tau\|_{L^{2q}(\gO;\bR)}\\
	&\qquad\cdot\big\|\max_{i=1,\dots,\tau}\|a_{N,\eps}\|_{\cW^{1,\infty}(\cT_i)}\big\|_{L^{4q}(\gO;\bR)}\|u_{N,\eps}\|_{L^{r_1}(\gO;L^2(\bT;V))}\Big)\\
	&\le C\big(\|\widetilde f\|_{L^{r_1}(\gO;L^2(\bT;H)}+\|u_{N,\eps}\|_{L^{r_1}(\gO;L^2(\bT;V))}\big).
	\end{align*}
	In the derivation, we have used the H\"older exponents $r_1\in(r,p)$ and $q:=(1/r-1/r_1)^{-1}<+\infty$. 
	The last estimate holds due to Lemmas~\ref{lem:a2} and~\ref{lem:w1infa} and Assumption~\ref{ass:EV22}(iv).
	By definition
	\bee
	\widetilde f(\go,\cdot,t)=f(\go,\cdot,t)-b_{N,\eps}(\go,\cdot)\cdot\nabla u_{N,\eps}(\go,\cdot,t)-\partial_t u_{N,\eps}(\go,\cdot,t),
	\eee
	hence Theorem~\ref{thm:exis2}, Lemma~\ref{lem:w1infa} and Lemma~\ref{lem:dtu_H} yield
	\be\label{eq:u_H2_2}
	\bE\Big(\int_0^T\sum_{i=1}^\tau \|u_{N,\eps}\|_{H^{1+\gk}(\cT_i)}^2dt\Big)^{1/2}
	\le C\Big(\|u_0\|_{L^p(\gO;V)}+\|f\|_{L^p(\gO;L^2(\bT;H))}\Big)<+\infty.
	\ee
\end{proof}

We are now ready to proof our main result:

\begin{proof}[Proof of Theorem~\ref{thm:semi_error_a}]

	We define the error $\theta_\ell:=u_{N,\eps}-u_{N,\eps,\ell}$ and observe that for fixed $\go\in\gO, t\in\bT$ Eqs.~\eqref{eq:semi_var} and~\eqref{eq:var2_approx} yield
	\begin{align*}
	\dualpair{V'}{V}{\partial_t\theta_\ell(\go,\cdot,t)}{v_{\ell,\go}}
	+B^{N,\eps}_\go(\theta_\ell(\go,\cdot,t),v_{\ell,\go})
	&=0\\
	\theta_\ell(\go,\cdot,0)&=(u_0-\cI_\ell u_0)(\go,\cdot),
	\end{align*}
	for all $v_{\ell,\go}\in V_\ell(\go)$.
	We then test against $v_{\ell,\go}=\cI_\ell u_{N,\eps}(\go,\cdot,t)-u_{N,\eps,\ell}(\go,\cdot,t)$ and integrate over $[0,t]$ to obtain
	\be\label{eq:sd1}
	\begin{split}
		\frac{1}{2}\|\theta_\ell(\go,\cdot,t)\|_H^2
		+\int_0^t(a_{N,\eps}(\go,\cdot),\sum_{i=1}^d(\partial_{x_i}(\theta_\ell(\go,\cdot,z))^2)dz= 
		&\frac{1}{2}\|\theta_\ell(\go,\cdot,0)\|_H^2\\
		&+\int_0^t\dualpair{V'}{V}{\partial_t\theta_\ell(\go,\cdot,z)}{(1-\cI_\ell)u_{N,\eps}(\go,\cdot,z)}dz\\
		&+\int_0^tB^{N,\eps}_\go(\theta_\ell(\go,\cdot,z),(1-\cI_\ell)u_{N,\eps}(\go,\cdot,z))dz\\
		&-\int_0^t(b_{N,\eps}(\go,\cdot)\cdot\nabla \theta_\ell(\go,\cdot,z),\theta_\ell(\go,\cdot,z))dz\\
		=&:\frac{1}{2}\|\theta_\ell(\go,\cdot,0)\|_H^2+I+II+III.
	\end{split}
	\ee
	Lemma~\ref{lem:dtu_H} implies that $\partial_t \theta_\ell(\go,\cdot,\cdot)\in L^2(\bT;H)$ and we use the Cauchy-Schwarz inequality to bound $I$:
	\bee
	I=\int_0^t(\partial_t\theta_\ell(\go,\cdot,z),(1-\cI_\ell)u_{N,\eps}(\go,\cdot,z))dz\le \int_0^t \|\partial_t\theta_\ell(\go,\cdot,z)\|_H\|(1-\cI_\ell)u_{N,\eps}(\go,\cdot,z))\|_Hdz.
	\eee
	We then use the Cauchy-Schwarz inequality and Ineq.~\eqref{eq:nabla_norm} to bound the second term 
	\begin{align*}
	II&=\int_0^t(a_{N,\eps}(\go,\cdot),\nabla \theta_\ell(\go,\cdot,z)\cdot\nabla(1-\cI_\ell)u_{N,\eps}(\go,\cdot,z))dz\\
	&\quad+\int_0^t(b_{N,\eps}(\go,\cdot)\cdot\nabla \theta_\ell(\go,\cdot,z),(1-\cI_\ell)u_{N,\eps}(\go,\cdot,z))dz\\
	&\le \int_0^t\Big(a_{N,\eps}(\go,\cdot)\big(\sum_{i=1}^d(\partial_{x_i}\theta_\ell(\go,\cdot,z))^2\big)^{1/2}
	,\,\big(\sum_{i=1}^d(\partial_{x_i}(1-\cI_\ell)u_{N,\eps}(\go,\cdot,z))^2\big)^{1/2}\Big)dz\\
	&\quad+\int_0^t2^{d/2-1/2}(\|b_{N,\eps}(\go,\cdot)\|_\infty\big(\sum_{i=1}^d(\partial_{x_i}\theta_\ell(\go,\cdot,z))^2\big)^{1/2},|(1-\cI_\ell)u_{N,\eps}(\go,\cdot,z)|)dz,
	\end{align*}
	and Young's inequality yields
	\begin{align*}
	II&\le \int_0^t\frac{1}{4}(a_{N,\eps}(\go,\cdot),\sum_{i=1}^d(\partial_{x_i}\theta_\ell(\go,\cdot,z))^2)
	+a_{N,\eps,+}(\go)|(1-\cI_\ell)u_{N,\eps}(\go,\cdot,z)|_{H^1(\cD)}^2dz\\
	&\quad+ \int_0^t\frac{1}{4}(a_{N,\eps}(\go,\cdot),\sum_{i=1}^d(\partial_{x_i}\theta_\ell(\go,\cdot,z))^2)
	+2^{d-1}\,\ol b_1^2 a_{N,\eps,+}(\go)\|(1-\cI_\ell)u_{N,\eps}(\go,\cdot,z)\|_H^2dz\\
	&\le \frac{1}{2}\int_0^t(a_{N,\eps}(\go,\cdot),\sum_{i=1}^d(\partial_{x_i}\theta_\ell(\go,\cdot,z))^2)dz+
	Ca_{N,\eps,+}(\go)\int_0^t\|(1-\cI_\ell)u_{N,\eps}(\go,\cdot,z)\|_{V}^2dz.
	\end{align*}
	Similarly, we bound the last term by
	\begin{align*}
	|III|\le \frac{1}{4}\int_0^t(a_{N,\eps}(\go,\cdot),\sum_{i=1}^d(\partial_{x_i}\theta_\ell(\go,\cdot,z))^2)dz+
	2^{d-1}\ol b_1\ol b_2 \int_0^t\|\theta_\ell(\go,\cdot,z)\|_H^2dz.
	\end{align*}
	We now plug in the estimates for $I-III$ in Eq.~\eqref{eq:sd1} and proceed in the fashion of Theorem~\ref{thm:exis} with Gr\"onwalls inequality to arrive at
	\begin{align*}
	\sup_{t\in\bT}\|\theta_\ell\|^2_{t,*}&\le C(1+1/a_{N,\eps,-}(\go))\big(\|\theta_\ell(\go,\cdot,0)\|_H^2
	+\|\partial_t\theta_\ell(\go,\cdot,\cdot)\|_{L^2(\bT;H)}\|(1-\cI_\ell)u_{N,\eps}(\go,\cdot,\cdot)\|_{L^2(\bT;H)}\\
	&\qquad+a_{N,\eps,+}(\go)\|(1-\cI_\ell)u_{N,\eps}(\go,\cdot,\cdot)\|_{L^2(\bT;V)}^2\big). 
	\end{align*}
	
	By Lemma~\ref{lem:dtu_H} it holds that $\partial_t\theta_\ell\in L^r(\gO;L^2(\bT;H)$ for any $r\in(2,p)$ with $p$ as in Assumption~\ref{ass:EV22}(iii).
	Taking expectations, using H\"older's inequality and $1/a_{N,\eps,-},a_{N,\eps,+}\in L^q(\gO;\bR)$ for all $q\ge 1$ then yields 
	\begin{align*}
	\bE(\sup_{t\in\bT}\|\theta_\ell\|^2_{t,*})^{1/2}
	&\le C \Big(\|\theta_\ell(\go,\cdot,0)\|_{L^p(\gO;H)}
	+\|(1-\cI_\ell)u_{N,\eps}(\go,\cdot,\cdot)\|_{L^r(\gO;L^2(\bT;H))}^{1/2}\\
	&\qquad+\|(1-\cI_\ell)u_{N,\eps}(\go,\cdot,\cdot)\|_{L^r(\gO;L^2(\bT;V))}\Big).
	\end{align*}
	By Assumption~\ref{ass:EV22}(iii) $u_0\in L^p(\gO;H^{1+\epsilon}(\cD))$ for some $\epsilon>0$,  which yields with Lemma~\ref{lem:interpolation}
	\bee
	\|\theta_\ell(\go,\cdot,0)\|_{L^p(\gO;H)}= \|(1-\cI_\ell)u_0\|_{L^p(\gO;H)}\le C\ol h_\ell^{1+\epsilon}.
	\eee
	Furthermore, Lemmas~\ref{lem:interpolation} and~\ref{lem:regularity} yield
	\bee
	\|(1-\cI_\ell)u_{N,\eps}(\go,\cdot,\cdot)\|_{L^r(\gO;L^2(\bT;H))}^{1/2}
	\le C\ol h_\ell^{(\gk+1)/2}\bE\Big(\big(\int_0^T\sum_{i=1}^\tau \|u_{N,\eps}\|_{H^{1+\gk}(\cT_i)}^2dt\big)^{r/2}\Big)^{1/2r}
	\le C\ol h_\ell^{(\gk+1)/2}
	\eee
	as well as 
	\bee
	\|(1-\cI_\ell)u_{N,\eps}(\go,\cdot,\cdot)\|_{L^r(\gO;L^2(\bT;V))}\le C\ol h_\ell^\gk \bE\Big(\big(\int_0^T\sum_{i=1}^\tau \|u_{N,\eps}\|_{H^{1+\gk}(\cT_i)}^2dt\big)^{r/2}\Big)^{1/2r}\le C\ol h_\ell^\gk.
	\eee
	The claim now follows since $0<\gk$, $\ol h_\ell\le 1$.
\end{proof}

\begin{rem}\label{rem:fem_rate}
	To ensure that the convergence of order $\ol h_\ell^\gk$ in Theorem~\ref{thm:semi_error_a} is not affected by the Gaussian field $W$, Assumption~\ref{ass:EV22}(i) cannot be relaxed. 
	For instance, given that $2\ga>\gb$, it follows from \cite[Proposition 3.4]{C12} that $a$ is piecewise H\"older-continuous with exponent $\varrho<\gb/2\ga$ and we may only expect a rate of order $\ol h_\ell^{\min(\varrho,\gk)}$, see \cite[Section 3]{CST13}, \cite[Section 5]{G13} and \cite[Chapter 10.1]{H10}.  
	In fact, we discuss an example with $\varrho=1/2-\epsilon$ in Section~\ref{sec:num2} and show that we only achieve a convergence rate of approximately $\ol h_\ell^{1/2}$ even for sample-adapted FE.
\end{rem}

\subsection{Temporal discretization}\label{subsec:temp_disc}
In the remainder of this section, we introduce a stable temporal discretization for the semi-discrete Problem~\eqref{eq:semi_var} and derive the corresponding mean-square error. To this end, we fix $\go\in\gO$ and let $u_{N,\eps,\ell}(\go,\cdot,\cdot)$ denote the sample-adapted semi-discrete approximation of $u_{N,\eps}(\go,\cdot,\cdot)$ from Eq.~\eqref{eq:semi_var}.
For a fully discrete formulation of Problem~\eqref{eq:semi_var}, we consider a time grid $0=t_0<t_1<\dots<t_n=T$ in $\bT$ for some $n\in\bN$. The temporal derivative at $t_i$ is approximated by the backward difference 
\bee
\partial_t u_{N,\eps,\ell}(\go,\cdot,t_i)\approx\frac{u_{N,\eps,\ell}(\go,\cdot,t_i)- u_{N,\eps,\ell}(\go,\cdot,t_{i-1})}{t_i-t_{i-1}},\quad i=1,\dots,n.
\eee
This yields the fully discrete problem to find $(u_{N,\eps,\ell}^{(i)}(\go,\cdot),\,i=0,\dots,n)\subset V_\ell(\go)$ such that for all $i=1,\dots,n$ and $v_{\ell,\go}\in V_\ell(\go)$
\be\label{eq:full_var}
\begin{split}
	\frac{1}{t_i-t_{i-1}}(u_{N,\eps,\ell}^{(i)}(\go,\cdot)-u_{N,\eps,\ell}^{(i-1)}(\go,\cdot),v_{\ell,\go})+B^{N,\eps}_{\go}( u_{N,\eps,\ell}^{(i)}(\go,\cdot),v_{\ell,\go})
	&=F_{t_i,\go}(v_{\ell,\go})\\
	u^{(0)}_{N,\eps,\ell}(\go,\cdot)&=\cI_\ell u_0(\go,\cdot).
\end{split}
\ee
For convenience, we assume an equidistant temporal grid with fixed time step $\gD t:=t_i-t_{i-1}>0$. 
The fully discrete solution is now given by  
\bee
u^{(i)}_{N,\eps,\ell}(\go,x)=\sum_{j=1}^{d_\ell}c_{i,j}(\go)v_{j,\go}(x),\quad i=1,\dots,n,
\eee
where the coefficient vector $\mathbf{c_i}(\go):=(c_{i,1}(\go),\dots,c_{i,d_\ell}(\go))^T$ solves the linear system of equations
\bee
(\mathbf M(\go)+\gD t\mathbf A(\go)) {\bf c_i(\go)}
=\gD t\mathbf F(\go,t_i)+\mathbf M(\go){\bf c_{i-1}(\go)}
\eee
in every discrete point $t_i$. The mass matrix consists of the entries $(\mathbf M(\go))_{jk}:=(v_{j,\go},v_{k,\go})$, the stiffness matrix and load vector are given by $(\mathbf A(\go))_{jk}=B^{N,\eps}_\go(v_{j,\go},v_{k,\go})$ and $(\mathbf F(\go,t_i))_j=F_{t_i,\go}(v_{j,\go})$ for $j,k\in\{1,\dots,d_\ell(\go)\}$, respectively, as in the semi-discrete case. 
The initial vector $c_0$ consists of the basis coefficients of $\cI_\ell u_0\in V_\ell$ with respect to $\{v_{1,\go},\dots,v_{d_\ell(\go),\go}\}$.
To extend the fully discrete solution $(u_{N,\eps,\ell}^{(i)}(\go,\cdot),i=0,\dots,n)$ to $\bT$, we define the linear interpolation 
\bee
\ol u_{N,\eps,\ell}(\go,\cdot,t):=(u_{N,\eps,\ell}^{(i)}(\go,\cdot)-u_{N,\eps,\ell}^{(i-1)}(\go,\cdot))\frac{(t-t_{i-1})}{\gD t}+ u_{N,\eps,\ell}^{(i-1)}(\go,\cdot),\quad t\in[t_{i-1},t_i],\quad i=1,\dots,n,
\eee
and are, therefore, able to estimate the resulting error with respect to the parabolic norm.
\begin{thm}\label{thm:full_error_a}
	Let Assumption~\ref{ass:EV22} hold, let $(u_{N,\eps,\ell}^{(i)},i=0,\dots,n)$ be the fully discrete sample-adapted approximation of $u_{N,\eps}$ as in Eq.~\eqref{eq:full_var} and let $\ol u_{N,\eps,\ell}$ be the linear interpolation in $\bT$.
	Then, 
	\begin{align*}
	\bE\big(\sup_{t\in\bT}\|u_{N,\eps,\ell}-\ol u_{N,\eps,\ell}\|^2_{*,t}\big)^{1/2}\le C\gD t.
	\end{align*}
\end{thm}

\begin{proof}
	We start by investigating the temporal regularity of $u_{N,\eps,\ell}$. For fixed $\go\in\gO$ and $0\le t_{i-1}<t_i\le T$ note that $w_i(\go,\cdot,t):=u_{N,\eps,\ell}(\go,\cdot,t)-u_{N,\eps,\ell}(\go,\cdot,t_{i-1})$ solves the variational problem 
	\begin{align*}
	&\dualpair{V'}{V}{\partial_t w_i(\go,\cdot,t)}{v_{\ell,\go}}
	+B^{N,\eps}_\go(w_i(\go,\cdot,t),v_{\ell,\go})
	=\dualpair{V'}{V}{f(\go,\cdot,t)- f(\go,\cdot,t_{i-1})}{v_{\ell,\go}}
	\end{align*}
	for $t\in[t_{i-1},t_i]$ and $v_{\ell,\go}\in V_\ell(\go)$ with initial condition $w(\go,\cdot,t_{i-1})=0$.
	Therefore, in the fashion of Theorem~\ref{thm:exis}, we obtain the pathwise parabolic estimate
	\be\label{eq:li_est}
	\begin{split}
		\sup_{t\in[t_{i-1},t_i]}\|w_i(\go,\cdot,t)\|^2_H+\int_{t_{i-1}}^{t}|w_i(\go,\cdot,z)|_{H^1(\cD)}^2dz
		\le &C(1+1/a_{N,\eps,-}(\go))\int_{t_{i-1}}^{t_i}\|f(\go,\cdot,z)-f(\go,\cdot,t_{i-1})\|^2_Hdz\\
		=&C(1+1/a_{N,\eps,-}(\go))\int_{t_{i-1}}^{t_i}\|\int_z^{t_i}\partial_t f(\go,\cdot,\widetilde z)d\widetilde z\|^2_{H}dz\\
		=&C(1+1/a_{N,\eps,-}(\go))\\
		&\quad\cdot\int_{t_{i-1}}^{t_i} \|\int_{t_{i-1}}^{t_i}\1_{[t_{i-1},t]}(\widetilde z)\partial_t f(\go,\cdot,\widetilde z)d\widetilde z\|^2_{H}dz\\
		\le&C(1+1/a_{N,\eps,-}(\go))\\
		&\quad\cdot\int_{t_{i-1}}^{t_i} (z-t_{i-1})dz\int_{t_{i-1}}^{t_i}\|\partial_t f(\go,\cdot,z)\|^2_{H}dz\\
		=&C(1+1/a_{N,\eps,-}(\go)) \frac{\gD t^2}{2}\|\partial_t f(\go,\cdot,\cdot)\|^2_{L^2([t_i,t_{i-1}];H)}.
	\end{split}
	\ee
	For the first identity we have used Lemma~\ref{lem:a2bs_calc}, the second estimate follows with H\"older's inequality. 
	Now let $\ol{\ol u}_{N,\eps,\ell}$ be the linear interpolation of the semi-discrete solution $u_{N,\eps,\ell}$ at the nodes $t_0,\dots,t_n$ and consider the splitting 
	\begin{align*}
	\bE\big(\sup_{t\in\bT}\|u_{N,\eps,\ell}-\ol u_{N,\eps,\ell}\|^2_{*,t}\big)^{1/2}
	&\le \bE\big(\sup_{t\in\bT}\|u_{N,\eps,\ell}-\ol{\ol u}_{N,\eps,\ell}\|^2_{*,t}\big)^{1/2}+\bE\big(\sup_{t\in\bT}\|\ol{\ol u}_{N,\eps,\ell}-\ol u_{N,\eps,\ell}\|^2_{*,t}\big)^{1/2}=:I+II.
	\end{align*}
	By Ineq.~\eqref{eq:li_est} it follows that 
	\begin{align*}
	\sup_{t\in\bT}\|u_{N,\eps,\ell}-\ol{\ol u}_{N,\eps,\ell}\|^2_{*,t}
	&\le \max_{i=1,\dots,n}\sup_{t\in[t_{i-1},t_i]}\|w_i(\go,\cdot,t)\|^2_H
	+2\sum_{i=1}^n \int_{t_{i-1}}^{t_{i}}\|w_i(\go,\cdot,t)\|_{H^1(\cD)}^2dt\\
	&\le C(1+1/a_{N,\eps,-}(\go)) \frac{\gD t^2}{2}\|\partial_t f(\go,\cdot,\cdot)\|^2_{L^2(\bT;H)}.
	\end{align*}
	By Assumption~\ref{ass:EV22}, H\"older's inequality and Lemma~\ref{lem:a2} with $q=(1/2-1/p)^{-1}$  
	\bee
	I\le C\gD t (1+\|1/a_{N,\eps,-}\|_{L^q(\gO;\bR)})\|\partial_t f\|_{L^p(\gO;L^2(\bT;H))}\le C\gD t.
	\eee
	
	Now let $\theta^{(i)}(\go,\cdot):=u_{N,\eps,\ell}(\go,\cdot,t_i)-u_{N,\eps,\ell}^{(i)}(\go,\cdot)$ denote the pathwise time discretization error at $t_i$.
	For any $t\in [t_{i-1},t_i]$, we observe that $(\ol u_{N,\eps,\ell}-\ol{\ol u}_{N,\eps,\ell})(\cdot,\cdot,t)$ is a convex combination of $\theta_i$ and $\theta_{i-1}$, and it holds that 
	\be\label{eq:li_est2}
	II\le \bE(\max_{i=1,\dots,n} \|\theta^{(i)}\|^2_H+\gD t\sum_{j=1}^i|\theta^{(j)}|^2_{H^1(\cD)})^{1/2}.
	\ee
	Hence, it is sufficient to control the errors at each $t_i$.
	Combining Eq.~\eqref{eq:full_var} and Eq.~\eqref{eq:semi_var} yields, for $i=1,\dots,n$, 
	\begin{align*}
	&\dualpair{V'}{V}{\theta^{(i)}(\go,\cdot)-\theta^{(i-1)}(\go,\cdot)}{v_{\ell,\go}}
	+\int_{t_{i-1}}^{t_i}B^{N,\eps}_\go(\theta^{(i)}(\go,\cdot),v_{\ell,\go})dt\\
	=&\int_{t_{i-1}}^{t_i}B^{N,\eps}_\go(u_{N,\eps,\ell}(\go,\cdot,t_i)-u_{N,\eps,\ell}(\go,\cdot,t),v_{\ell,\go})
	+\dualpair{V'}{V}{f(\go,\cdot,t)- f(\go,\cdot,t_i)}{v_{\ell,\go}}dt\\
	:=&\int_{t_{i-1}}^{t_i} \dualpair{V'}{V}{\ol f_i(\go,\cdot,t)}{v_{\ell,\go}}dt
	\end{align*}
	and initial condition $u_{N,\eps,\ell}(\go,\cdot,0)-u_{N,\eps,\ell}^{(0)}(\go,\cdot)=0$.
	We now test against $v_{\ell,\go}=\theta^{(i)}(\go,\cdot)$, sum until $i$ and use the discrete Gr\"onwall inequality to obtain 
	(as in Theorem~\ref{thm:exis}) the discrete estimate 
	\begin{align*}
	\max_{i=1,\dots,n} \|\theta^{(i)}(\go,\cdot)\|^2_H+\gD t\sum_{j=1}^i|\theta^{(j)}(\go,\cdot)|^2_{H^1(\cD)}
	&\le C(1+1/a_{N,\eps,-}(\go)) \sum_{i=1}^n \|\ol f_i(\go,\cdot,\cdot)\|^2_{L^2([t_i,t_{i-1}];V')}\\
	&\le C(1+1/a_{N,\eps,-}(\go)) \\
	&\quad \cdot \Big(\sum_{i=1}^n \int_{t_{i-1}}^{t_i} |u_{N,\eps}(\cdot,\cdot,t)-u_{N,\eps}(\cdot,\cdot,t_i)|^2_{H^1(\cD)}dt\\
	&\qquad +\sum_{i=1}^n\|f(\go,\cdot,t)-f(\go,\cdot,t_1)\|^2_Hdt\Big).
	\end{align*}
	Proceeding as for Ineq.~\eqref{eq:li_est}, this implies 
	\begin{align*}
	\max_{i=1,\dots,n} \|\theta^{(i)}(\go,\cdot)\|^2_H+\gD t\sum_{j=1}^i|\theta^{(j)}(\go,\cdot)|^2_{H^1(\cD)}
	\le C(1+1/a_{N,\eps,-}(\go))a_{N,\eps,+}(\go)^2 \gD t^2 \|\partial_t f(\go,\cdot,\cdot)\|^2_{L^2(\bT;H)}.
	\end{align*}
	We use Assumption~\ref{ass:EV22}, H\"older's inequality and Lemma~\ref{lem:a2} 
	\bee
	\bE\big(\max_{i=1,\dots,n} \|\theta^{(i)}\|^2_H+\gD t\sum_{j=1}^i|\theta^{(j)}|^2_{H^1(\cD)}\big)^{1/2}
	\le C\gD t \|\partial_t f\|_{L^p(\gO;L^2(\bT;H))}
	\le C\gD t,
	\eee 
	and the claim finally follows by Ineq.~\eqref{eq:li_est2}.
\end{proof}

To conclude this section, we record a bound on the overall approximation error, which is an immediate consequence of Theorems~\ref{thm:u_error2}, \ref{thm:semi_error_a} and \ref{thm:full_error_a}.
\begin{cor}\label{cor:total_error}
	Let Assumptions~\ref{ass:EV22} and~\ref{ass:fe} hold and let $\ol u_{N,\eps,\ell}$ be the linear interpolation of the fully discrete approximation of $(u^{(i)}_{N,\eps},i=0,\dots,n)$. Then, 
	\begin{align*}
	\bE\big(\sup_{t\in\bT}\|u-\ol u_{N,\eps,\ell}\|^2_{*,t}\big)^{1/2}\le C\big(\Xi_N^{1/2}+\eps^{1/s}+\overline h_\ell^\gk+\gD t\big).
	\end{align*}
\end{cor}

\section{Numerical experiments} \label{sec:num2}
In all of our numerical experiments we measure the root mean-square error
\bee
RMSE:=\bE(\|u(\cdot,\cdot,T)-\ol u_{N,\eps,\ell}(\cdot,\cdot,T)\|_V^2)^{1/2}.
\eee
For each given FE discretization parameter $\ol h_\ell$, we align the error contributions of $N,\eps$ and $\gD t$ such that $\Xi_N^{1/2}\simeq\eps^{1/s}\simeq\gD t\simeq \ol h_\ell$. 
Hence, the dominant source of error is the spatial discretization and Corollary~\ref{cor:total_error} yields $RMSE\le C\ol h_\ell^\gk$. 
This allows us to measure the value of $\gk$ in Assumption~\ref{ass:fe} by linear regression.
While the choices of $\gD t$ and $\eps$ are usually straightforward for given $\ol h_\ell$, we refer to \cite[Remark 5.3]{BS18b} where we describe how to achieve $\Xi_N^{1/2}\simeq\ol h_\ell$ for common examples of covariance operators $Q$.
To emphasize the advantage of the sample-adapted FE algorithm introduced in Section~\ref{sec:fem2}, we also repeat all experiments with a standard FE approach and compare the resulting errors. For the non-adapted FE algorithm, we use for a given triangulation diameter $h_\ell$ the same approximation parameters $\gD t, N$ and $\eps$ as for the corresponding sample-adapted method. This ensures that the weaker performance of this non-adapted method is due to the mismatch between FE triangulation and the discontinuities of $a$ and $b$.
We approximate the entries of the stiffness matrix for both FE approaches by the midpoint rule on each interval (in 1D) or triangle (in 2D).
If the triangulation is aligned to the discontinuities in $a$ and $b$, this adds an additional non-dominant term of order $\overline h_\ell$ to the error estimate in Corollary~\ref{cor:total_error}, see for instance \cite[Prop. 3.13]{CST13}.  
Thus, the bias stemming from the midpoint rule does not dominate the overall order of convergence in the sample-adapted algorithm.
In the other case, we cannot quantify the quadrature error due to the discontinuities on certain triangles but suggest an error of order $\ol h_\ell^{1/2}$ based on our experimental observations.

\subsection{Numerical examples in 1D}
For the first scenario in this subsection, we consider the advection-diffusion problem~\eqref{eq:pde} on the domain $\cD=(0,1)$, with $T=1$, $u_0(x)=\sin(\pi x)/10$ and source term $f\equiv1$. 
The continuous part of the jump-diffusion coefficient $a$ is given by $\ol a\equiv 0$ and $\Phi(w)=\exp(w)$, where the Gaussian field $W$ is characterized by the \textit{Mat\'ern covariance operator} 
\bee
Q_{M}:H\to H, \quad [Q_{M}\varphi](y):=\int_\cD \sigma^2\frac{2^{1-\nu}}{\gG(\nu)}\Big(\sqrt{2\nu}\frac{|x-y|}{\delta}\Big)^\nu K_{\nu}\Big(\sqrt{2\nu}\frac{|x-y|}{\delta}\Big)\varphi(x)dx\quad\text{for $\varphi\in H$},
\eee
with smoothness parameter $\nu>0$, variance $\sigma^2>0$ and correlation length $\delta>0$. 
Above, $\gG$ denotes the Gamma function and $K_\nu$ is the modified Bessel function of the second kind with $\nu$ degrees of freedom.
It is known that $W$ is mean square differentiable if $\nu>1$ and, moreover, the paths of $W$ are almost surely in $C^{\floor{\nu},\varrho}(\cD;\bR)$ with $\varrho<\nu-\floor{\nu}$ for any $\nu\ge 1/2$, see \cite[Section 2.2]{Gr15}.
The spectral basis of $Q_{M}$ may be efficiently approximated by Nystr\"om's method, see for instance \cite{SP07}.
In our experiments, we use the parameters $\nu=3/2,\, \sigma^2=1$ and $\delta=0.05$.

For each experiment in one dimension, the number of partition elements is given by $\tau=\cP+2$, where $\cP$ is Poisson-distributed with intensity parameter $5$.  
On average, this splits the domain in $7$ disjoint intervals and the diffusion coefficient has almost surely at least one discontinuity.
The position of each jump is sampled according to the measure $\gl$, which we set as the Lebesgue measure $\gl_L$ on $(\cD,\cB(\cD))$.
More precisely, let $(\widetilde x_i,i\in\bN)$ be an i.i.d. sequence of $\cU(\cD)$-random variables that are independent of $\tau$.
We take the first $\tau-1$ points of this sequence, order them increasingly and denote the ordered subset by $0<x_1<\dots<x_{\tau-1}<1$.  
This generates the random partition $\cT=\{(0,x_1),(x_1,x_2),\dots,(x_{\tau-1},1)\}$ for each realization of $\tau$.
Conditional on the random variable $\tau=\cP+2\ge2$, the distribution of each $x_i$ for $i=1,\dots,\tau-1$ is then given by
\bee
\bP(x_i\le c\,|\tau)=\frac{(\tau-1)!}{(\tau-i)!(i-1)!}c^{\tau-i}(1-c)^{i-1},\quad c\in\cD=(0,1).
\eee
This can be utilized to derive further statistics, such as the average interval width of $\cT$ given by
\bee
\bE(\bE(x_1|\tau))=\bE\big(\int_0^1c^{\tau-1}dc\big)=\bE(1/\tau)=\sum_{k=0}^\infty \frac{5^ke^{-5}}{k!}\frac{1}{k+2}\approx 0.1603
\eee
with corresponding variance $\bE(1/(\tau+1))-\bE(1/\tau)^2\approx 0.1102$.
This also shows that increasing the Poisson parameter in $\cP$ resp. $\tau$ would yield a longer average computational time, as more and smaller intervals would be sampled.
The order of spatial convergence of the sample-adapted FE scheme on the other hand remains unaffected of the distribution of $\cT$.
In the subsequent examples we vary the distribution of the jump heights $P_i$ and use the jump-advection coefficient given by 
\bee
b(\go,x):=2\sin(2\pi x)a(\go,x),\quad \go\in\gO,\;x\in\cD.
\eee
Note that we did not impose an upper deterministic bound $\ol b_2$ on $b$. 
To obtain pathwise approximations of the samples $u_{N,\eps}(\go,\cdot,\cdot)$, we use non-adapted and sample-adapted piecewise linear elements and compare both approaches. 
The FE discretization parameter is given by $\ol h_\ell=2^{-\ell}/4$ and we consider the range $\ell=1,\dots,6$. 
We approximate the reference solution $u$ for each sample using sample-adapted FE and set $u_{ref}:=\ol u_{N_8,\eps_8,8}(\cdot,\cdot,T)$, where we choose $\gD t_8\simeq\Xi_{N_8}^{1/2}\simeq\eps_8^{1/2}\simeq 2^{-10}$. The RMSE is estimated by averaging 100 samples of $\|u_{ref}-\ol u_{N,\eps,\ell}(\cdot,\cdot,T)\|_V^2$ for $\ell=1,\dots,6$. 
To subtract sample-adapted\textbackslash non-adapted approximations from the reference solution $u_{ref}$, we use a fixed grid with $2^{10}+1$ equally spaced points in $\cD$, thus the error stemming from interpolation\textbackslash prolongation may be neglected.
Given that $RMSE\approx C\ol h_\ell^\gk$, it holds that 
\bee
\log(RMSE)\approx \gk\log(\ol h_\ell)+\log(C)
\eee
and we estimate the convergence rate $\gk$ by a linear regression of the log-RMSE on the log-refinement sizes $\log(\ol h_\ell)$. 
As we consider 1D-problems in this subsection, we expect convergence rates close to one for the sample-adapted method whenever Assumption~\ref{ass:EV22} holds.

\begin{figure}[ht]
	\centering
	\subfigure{\includegraphics[scale=0.36]{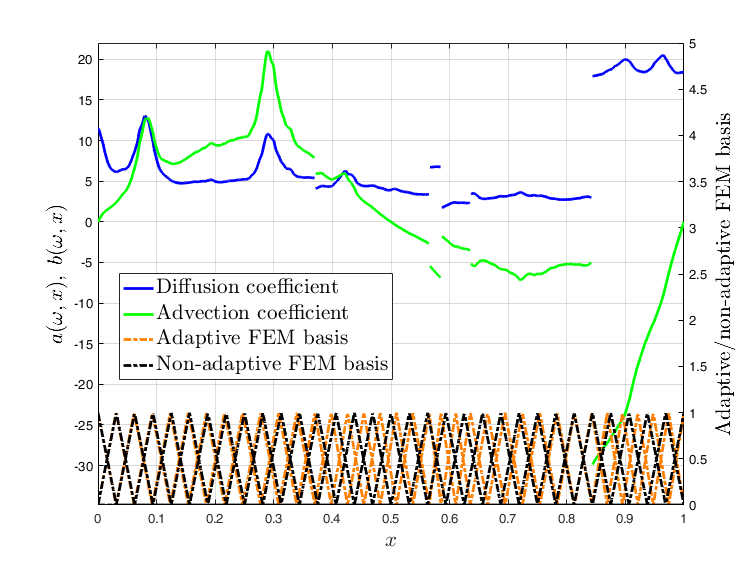}}
	\subfigure{\includegraphics[scale=0.36]{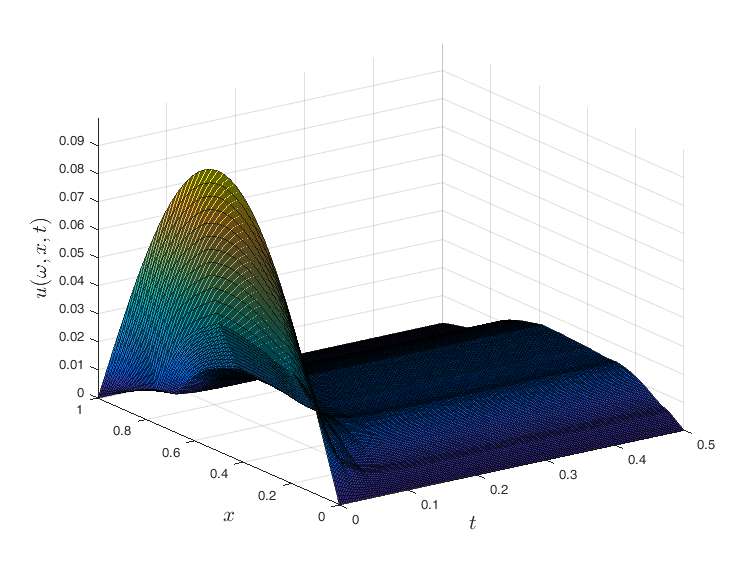}}
	\subfigure{\includegraphics[scale=0.36]{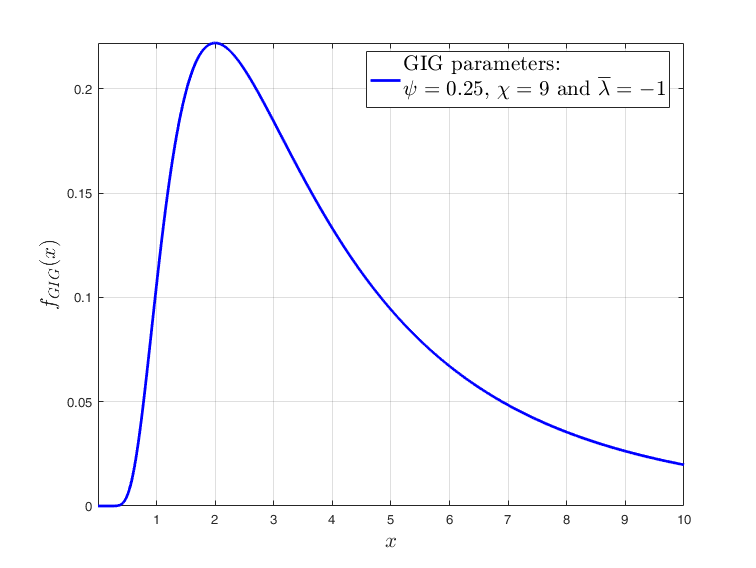}}
	\subfigure{\includegraphics[scale=0.36]{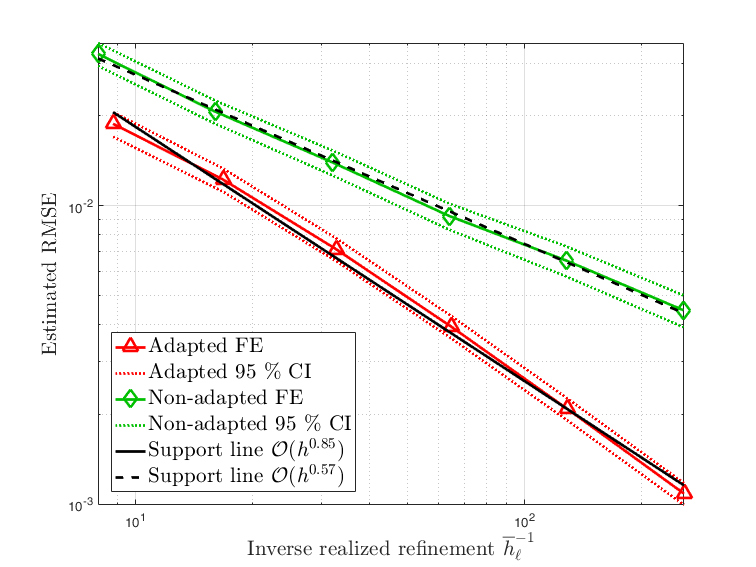}}
	\caption{First numerical example in 1D with Mat\'ern covariance operator and GIG distributed jumps. 
		Top left: Jump-diffusion/advection coefficient and adapted/non-adapted FE basis, top right: FE solution corresponding to the sample on the left and the given sample-adapted FE basis, bottom left: GIG density function and parameters,
		bottom right: estimated RMSE vs. inverse spatial refinement parameter size.}
	\label{fig:1D_GIG}
\end{figure}
In our first numerical example, the jump heights $P_i$ follow a \textit{generalized inverse Gaussian} (GIG) distribution with density
\bee
f_{GIG}(x)=\frac{(\psi/\chi)^{\gl/2}}{2K_{\gl}(\sqrt{\psi\chi})}x^{\gl-1}\exp\big(-\frac 1 2(\psi x+\chi x^{-1})\big),\quad x>0,
\eee
and parameters $\chi,\psi>0$, $\gl\in\bR$, see \cite{BN78}.
Unbiased sampling from this distribution may be rather expensive, hence we generate approximations $\widetilde P_i$ of $P_i$ by a Fourier inversion technique which guarantees that $\bE(|\widetilde P_i-P_i|^2)\le\eps$ for any desired $\eps>0$.
This allows us to adjust the sampling bias $\eps>0$ with $\ol h_\ell$ (and the corresponding $\gD t$ and $\Xi_N$) for any $\ell\in\bN_0$.
Details on the Fourier inversion algorithm, the sampling of GIG distributions and the corresponding $L^2(\gO;\bR)$-error may be found in~\cite{BS18a}.
The GIG parameters are set as $\psi=0.25, \chi=9$ and $\ol\gl=-1$, the resulting density and a sample of the coefficients are given in Fig.~\ref{fig:1D_GIG}. 
As expected, we see in Fig.~\ref{fig:1D_GIG}, that the sample-adapted algorithm converges with rate $\gk=0.85$. 
Thus, the sampling error of the GIG jump heights does not dominate the remaining error contributions. 
Compared to adapted FE, the non-adapted method converges at a significantly lower rate of $0.57$.

In Remark~\ref{rem:fem_rate}, we stated that the condition $2\ga\le\gb$ on the decay of the eigenvalues of $Q$ entails mean square differentiability of $W$ and thus a convergence rate of order $\gk$ in the sample-adapted method. We suggested that this rate will deteriorate if the paths of $W$ are only H\"older continuous with exponent $\varrho<\gk\le 1$. 
To illustrate this, we repeat the first experiment with a changed covariance operator.   
We now consider the \textit{Brownian motion covariance operator} 
\bee
Q_{BM}:H\to H,\quad [Q_{BM}\varphi](y):=\int_\cD min(x,y)\varphi(x)dx\quad\text{for $\varphi\in H$},
\eee
with eigenbasis given by $\eta_i=(2\sqrt 2/((2i+1)\pi))^2$ and $e_i(x)=\sin((2i+1)\pi x/2)$ for $i\in\bN_0$.
The paths of $W$ generated with $Q_{BM}$ are H\"older-continuous with $\varrho=1/2-\epsilon$ for any $\epsilon>0$ because $\gb=1-\epsilon$ and $\ga=1$. 
A sample of the coefficients is given in Fig.~\ref{fig:1D_Unif_BB}. The sample-adapted RMSE is smaller than the non-adapted curve and decreases slightly faster, but both errors now decay at a lower rate of roughly $\approx 1/2$ due to the lack of (piecewise) spatial regularity of $a$ and $b$. 
\begin{figure}[ht]
	\centering
	\subfigure{\includegraphics[scale=0.36]{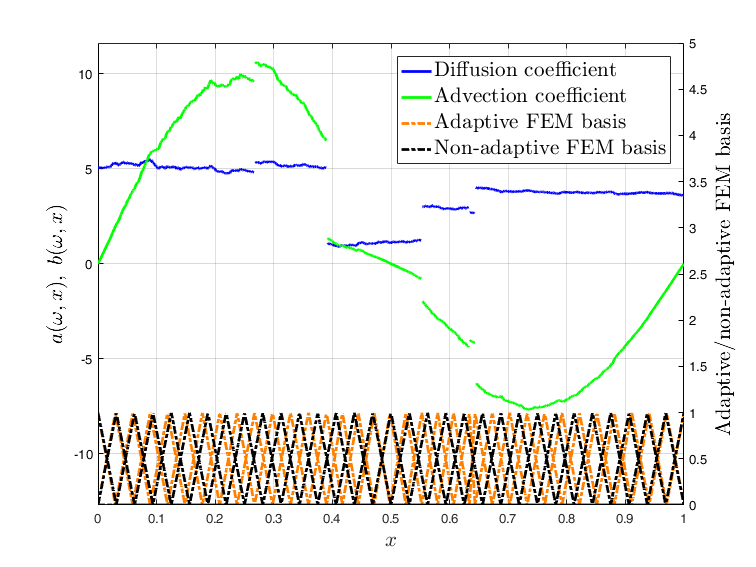}}
	\subfigure{\includegraphics[scale=0.36]{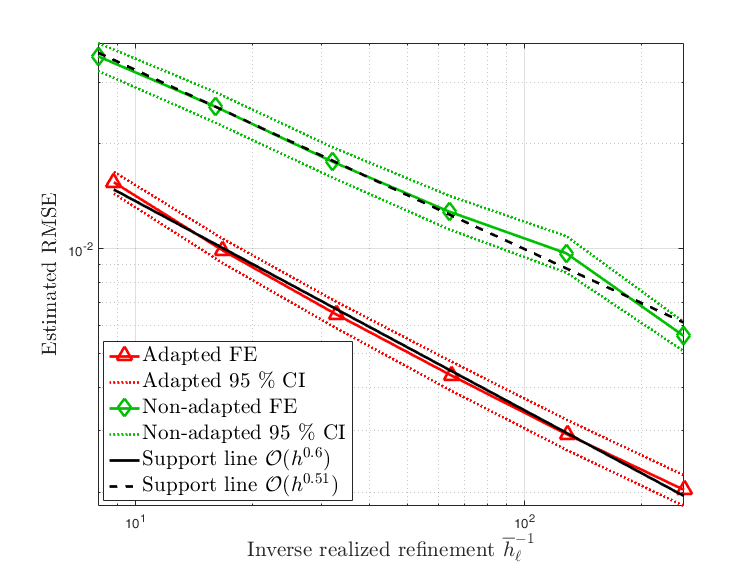}}
	\caption{Second numerical example in 1D with Brownian motion covariance operator and uniformly distributed jumps. 
		Top left: Jump-diffusion/advection coefficient and adapted/non-adapted FE basis, top right: FE solution corresponding to the sample on the left and the given sample-adapted FE basis, bottom: estimated RMSE vs. inverse spatial refinement size.}
	\label{fig:1D_Unif_BB}
\end{figure}

\subsection{Numerical examples in 2D}
In two spatial dimensions, we work on $\cD=(0,1)^2$ with $T=1$, initial data $u_0(x_1,x_2)=\frac{1}{100}\sin(\pi x_1)\sin(\pi x_2)$, source term $f\equiv1$ and assume that $\bar a\equiv0$. 
The Gaussian part of $a$ is determined by the \KL expansion
\bee
W(x)=\sum_{i\in\bN}\sqrt{\eta_i}e_i(x)Z_i,\quad x\in\cD,\quad Z_i\stackrel{i.i.d.}\sim\cN(0,1),
\eee
with spectral basis given by $\eta_i:=\sigma^2\exp(-\pi^2i^2\delta^2)$ and $e_i(x):=\sin(\pi ix_1)\sin(\pi ix_2)$ for $i\in\bN$.
Again, the parameters $\delta,\sigma^2>0$ denote the correlation length and total variance of $W$ respectively.
It can be shown that these eigenpairs solve the integral equation
\bee
\sigma^2\int_{\cD}\frac{1}{4\pi t}\exp\big(-\frac{-||x-y||_2^2}{2\delta^2}\big)e_i(y)dy=\eta_i e_i(x),\quad i\in\bN,
\eee
with $e_i=0$ on $\partial\cD$, see \cite{GN13}. 
Compared with a Gaussian field generated by a squared exponential covariance operator, this field shows a very similar behavior, except that it is zero on the boundary. It, further, has the advantage, that all expressions are available in closed form and we forgo the numerical approximation of the eigenbasis. The eigenvalues decay exponentially fast with respect to $i$, hence Assumption~\ref{ass:EV22} is fulfilled and we use the parameters $\sigma^2=0.25$ and $\delta=0.02$ for all experiments in this section.
As before, we consider a log-Gaussian random field, meaning $\Phi(w)=\exp(w)$.
To illustrate the flexibility of a jump-diffusion coefficient $a$ as in Def.~\ref{def:a2}, we vary the random partitioning of $\cD$ for each example and give a detailed description below.
We set the spatial discretization parameter to $\ol h_\ell=h_\ell=\frac{2}{5}2^{-\ell}$ and consider the cases $\ell=1,\dots,5$. To estimate the RMSE, we sample similar to the one-dimensional case the reference solution $u_{ref}:=\ol u_{N_7,\eps_7,7}(\cdot,\cdot,T)$ with $\gD t_7\simeq\Xi_{N_7}^{1/2}\simeq\eps_7^{1/2}\simeq \frac{2}{5}2^{-7}$ and average again 100 independent samples of $\|u_{ref}-\ol u_{N,\eps,\ell}(\cdot,\cdot,T)\|_V^2$. 
For interpolation/prolongation we use a reference grid with $(2^{8}+1)\times (2^{8}+1)$ equally spaced points in $\cD$.
The convergence rate, i.e. the exponent $\gk$ from Assumption~\ref{ass:fe}, in the sample-adapted method is estimated by linear regression as for the one-dimensional examples.
We further use the (unbounded) jump-advection coefficient
\bee
b(\go,x,y)=5\sin(\pi x)\sin(\pi y)a(\go,x,y) \begin{pmatrix}
	1\\ 1
\end{pmatrix},\quad \go\in\gO,\;(x,y)\in\cD
\eee
in each scenario.

In our first 2D example, we aim to imitate the structure of a heterogeneous medium. 
For this, we divide the domain by two horizontal and vertical lines. 
We assume that the horizontal resp. vertical lines do not intersect each other and thus obtain $\tau\equiv9$.
The remaining four intersection points of the lines in $\cD$ are uniformly distributed in $(0.2,0.8)^2$.
This is realized by setting $\gl$ as the two-dimensional Lebesgue-measure restricted to $(0.2,0.8)^2\subset\cD$.
Finally, we assign i.i.d. jump heights $P_i\sim\cU(0,10)$ to each partition element $\cT_i$.
Fig.~\ref{fig:2D_HM} shows a sample of the advection- and diffusion coefficient for the heterogeneous medium together with the associated (adapted) FE approximation of $u$. 
As before, the sample-adapted method is advantageous and the regression suggests that Assumption~\ref{ass:fe} holds with $\gk=0.86$.
If we use non-adapted FE, we may still recover a convergence rate of $0.66$, which is actually slightly better than the expected rate of $0.5$.
\begin{figure}[ht]
	\centering
	\subfigure{\includegraphics[scale=0.36]{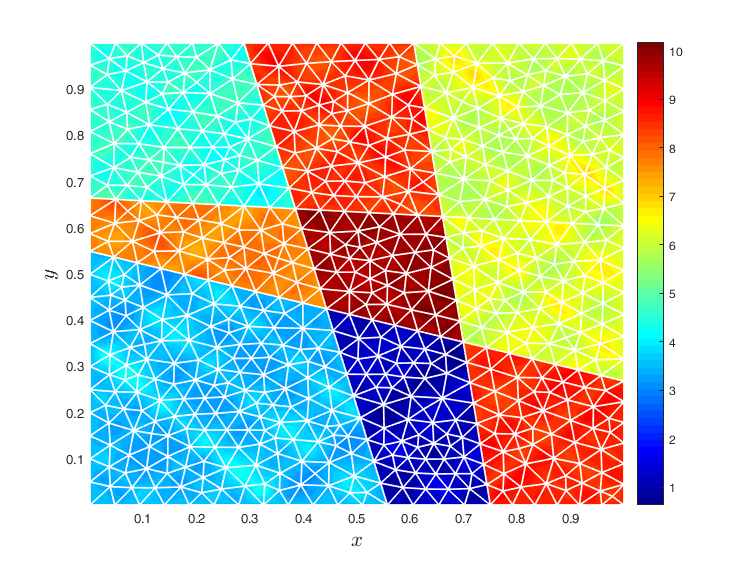}}
	\subfigure{\includegraphics[scale=0.36]{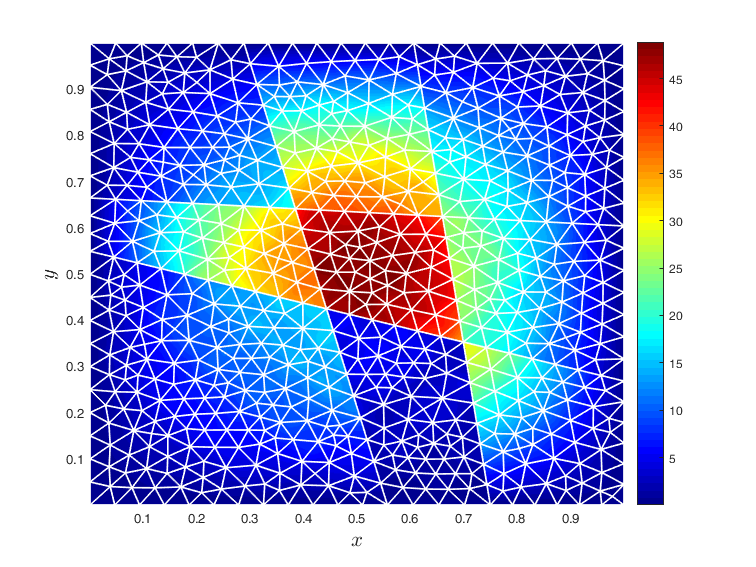}}
	\subfigure{\includegraphics[scale=0.36]{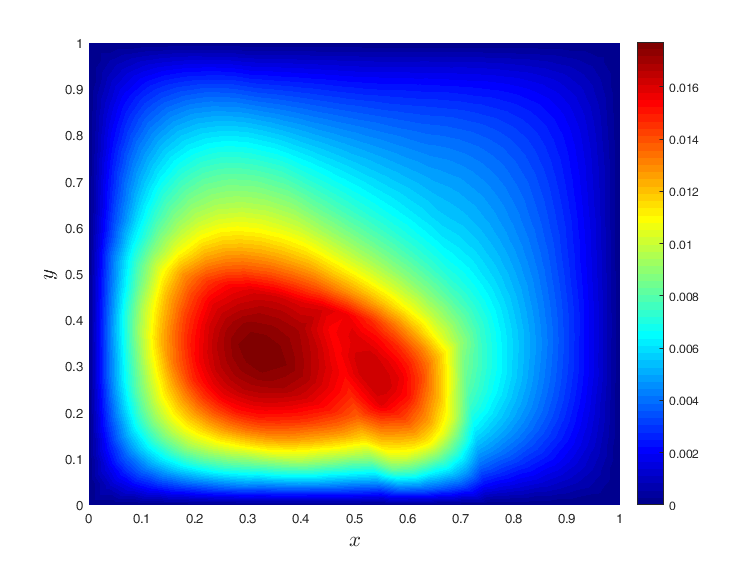}}
	\subfigure{\includegraphics[scale=0.36]{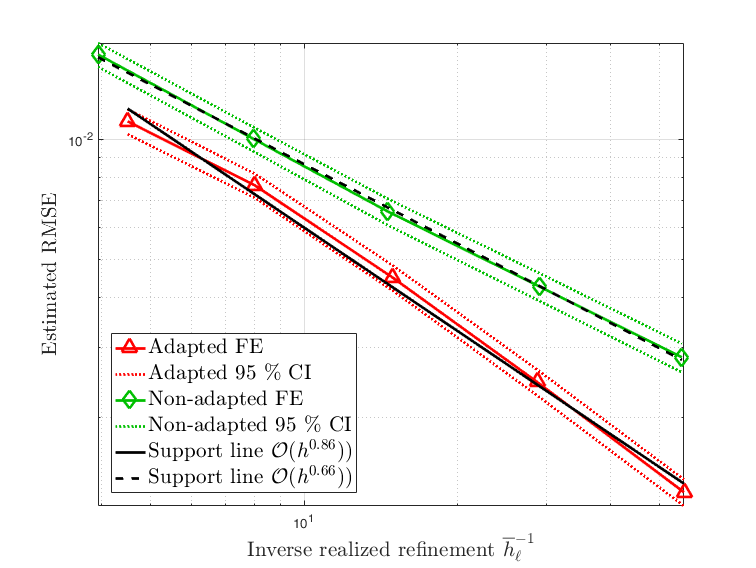}}
	\caption{First numerical example in 2D (heterogeneous medium). 
		Top left: sample of the jump-diffusion coefficient and sample-adapted triangulation, top right: sample of the jump-advection coefficient with adapted triangulation, bottom left: FE solution at $T$ corresponding to the samples and triangulations on the top, 
		bottom right: estimated RMSE vs. inverse spatial refinement $\ol h_l^{-1}$.}
	\label{fig:2D_HM}
\end{figure}

We now consider an example with lower expected regularity and pure jump field, i.e. $\ol a$ and $\Phi$ are set to zero. Therefore, we need to consider strictly positive jump heights $P_i$ to ensure well-posedness of the problem. 
We sample one $\cU([0.4,0.6]^2)$-distributed center point $x_c\in\cD$ and split the domain by a vertical and horizontal line through $x_c$. This yields a partition of $\cD$ into four squares $\cT_1-\cT_4$. We then sample a random variable $P_1\sim\cU([10^{-4},10^{-2}])$ and assign the value of $P_1$ to the lower left and the upper right partition element. The remaining elements are equipped with inverse value $P_2=P_1^{-1}$, see Fig.~\ref{fig:2D_squares} for a sample of the coefficients. From deterministic regularity theory, it is known that for given $P_1$ the solution to this problem has only $H^{1+\gk}$-regularity around $x_c$, where $\gk=\cO(P_1)$, see e.g. \cite{P01}.
Consequently, we see deteriorated convergence rates compared to the first example. The non-adapted method now performs poorly with an error decay of a rate less than $0.5$, whereas the sample-adapted method still recovers a rate of $0.7$. 
A possible explanation for this behavior is that the sample-adapted algorithm generates a mesh with respect to the singularity at $x_c$. Optimal meshes for this problem refine in the vicinity of $x_c$ and then coarsen on the interior of the partition elements, for instance \textit{graded meshes} or \textit{bisection meshes} as used in \cite{MS15} and the references therein. 

\begin{figure}[ht]
	\centering
	\subfigure{\includegraphics[scale=0.36]{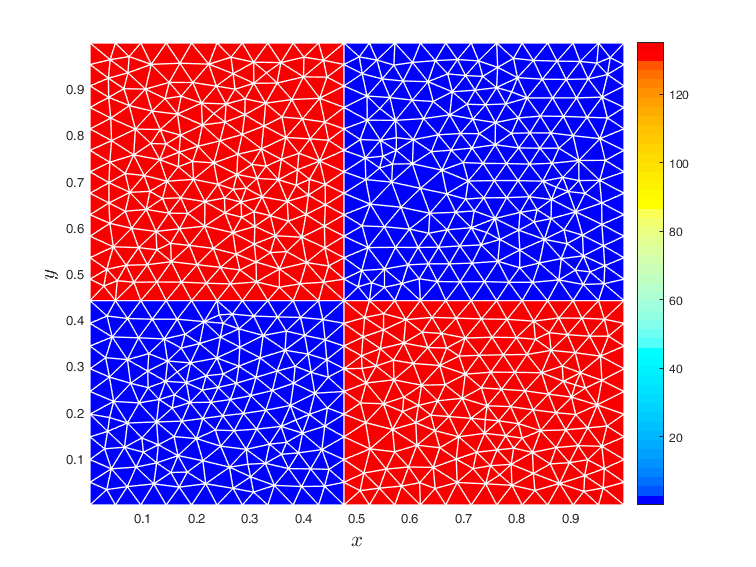}}
	\subfigure{\includegraphics[scale=0.36]{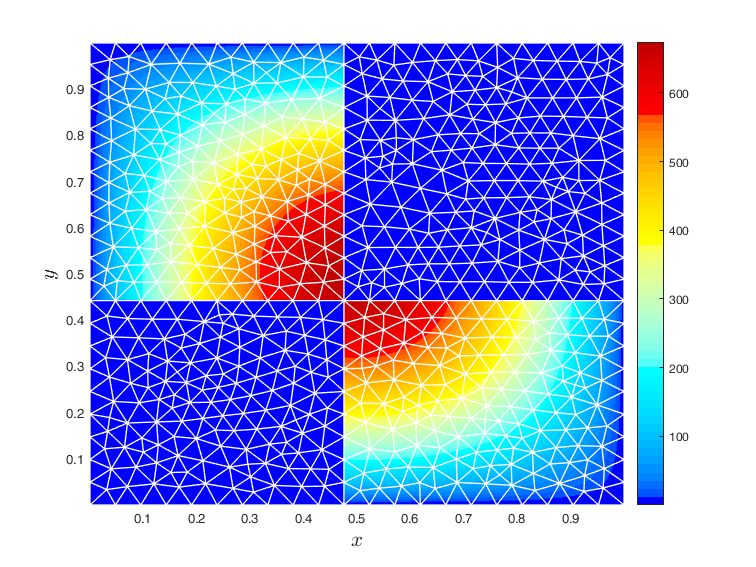}}
	\subfigure{\includegraphics[scale=0.36]{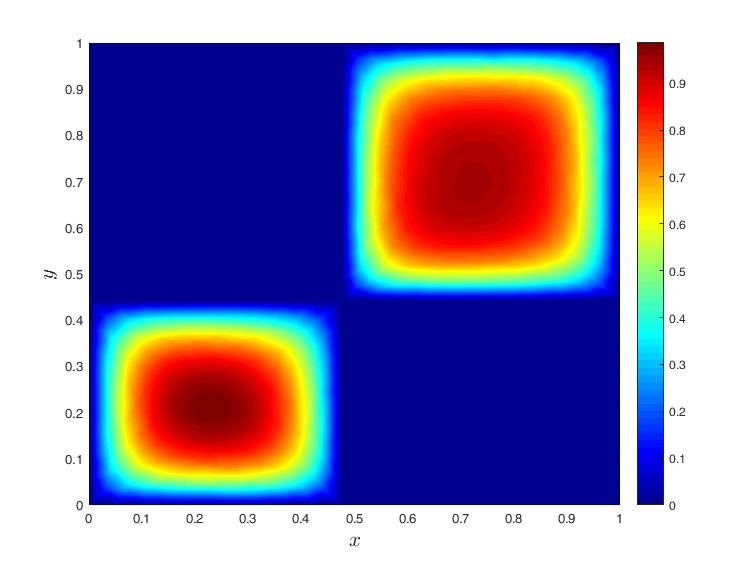}}
	\subfigure{\includegraphics[scale=0.36]{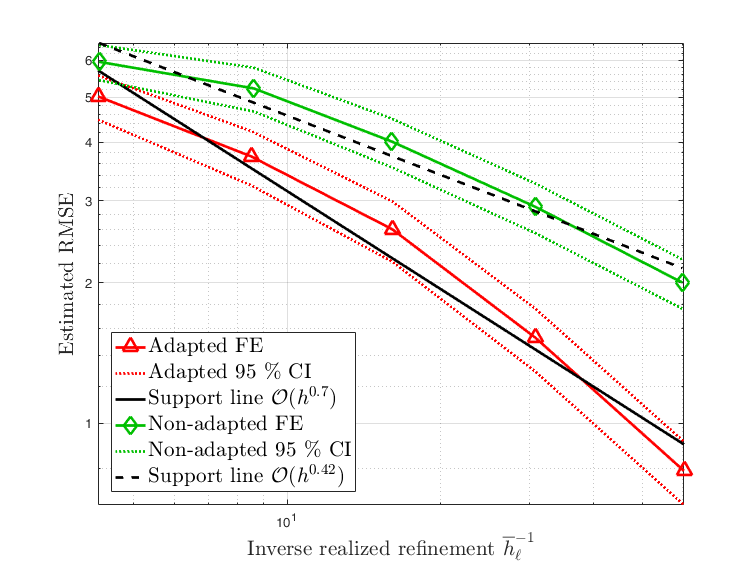}}
	\caption{Second numerical example in 2D. 
		Top left: sample of the jump-diffusion coefficient and sample-adapted triangulation, top right: sample of the jump-advection coefficient with adapted triangulation, bottom left: FE solution at $T$ corresponding to the samples and triangulations on the top, 
		bottom right: estimated RMSE vs. inverse spatial refinement parameter $\ol h_l^{-1}$.}
	\label{fig:2D_squares}
\end{figure}

To conclude, we suggest that a more effective refinement in two spatial dimensions may be achieved by \textit{h-Finite Element methods} (see \cite{S99}), i.e. by refining the sample-adapted mesh in the reentrant corners. 
A thorough analysis of this approach for general random geometries is subject to further research.   

\section*{Acknowledgements}

The research leading to these results has received funding from the Deutsche Forschungsgemeinschaft (DFG, German Research Foundation) under Germany's Excellence Strategy – EXC-2075 – 390740016 at the University of Stuttgart and it is gratefully acknowledged.
The authors would like to thank Prof. Dr. Christoph Schwab for his valuable suggestions that lead to a significant improvement of this manuscript.

\addcontentsline{toc}{section}{References}
\bibliographystyle{plain}
\bibliography{parabolic_jump_diffusion}

\end{document}